\documentclass[12pt]{article}

\usepackage[english]{babel}
\usepackage{amsmath, amsfonts, amssymb, amsthm}
\usepackage{url}
\usepackage[misc,geometry]{ifsym} 
\usepackage{enumitem}
\usepackage{fullpage}
\usepackage{mathtools}
\usepackage[T2A]{fontenc}
\usepackage[utf8]{inputenc}
\usepackage{graphicx}
\usepackage{subcaption}

\usepackage{natbib}

\usepackage[pdftex]{hyperref}
\usepackage[all]{hypcap}

\usepackage{algorithm}
\usepackage{algorithmic}
\usepackage{caption}

\newcommand{\R}{{\mathbb R}}

\newcommand{\e}{\varepsilon}
\newcommand{\la}{\langle}
\newcommand{\ra}{\rangle}
\def\vp{\varphi}

\newcommand\uprule{\rule{0mm}{1.9ex}}

\newtheorem{Lm}{Lemma}

\newtheorem{Th}{Theorem}

\newtheorem{Cor}{Corollary}

{
\theoremstyle{remark}

}
{
\theoremstyle{definition}
\newtheorem{Def}{Definition}

}

\title{Adaptive Similar Triangles Method: a Stable Alternative to Sinkhorn’s Algorithm for Regularized Optimal Transport} 
\author{Pavel Dvurechensky
				\thanks{Weierstrass Institute for Applied Analysis and Stochastics, Berlin; Institute for Information Transmission Problems RAS, 
								Moscow, pavel.dvurechensky@wias-berlin.de}
				\and
				Alexander Gasnikov
				\thanks{Moscow Institute of Physics and Technology, Moscow; Institute for Information Transmission Problems RAS, 
								Moscow, gasnikov@yandex.ru}
				\and				
				Sergey Omelchenko
				\thanks{Moscow Institute of Physics and Technology, Moscow, sergey.omelchenko@phystech.edu}
				\and
				Alexander Tiurin
				\thanks{National Research University Higher School of Economics, Moscow, alexandertiurin@gmail.com}
				}

\date{\today}

\begin{document}

\maketitle

\begin{abstract}
In this paper, we are motivated by two important applications: entropy-regularized optimal transport problem and road or IP traffic demand matrix estimation by entropy model. Both of them include solving a special type of optimization problem with linear equality constraints and objective given as a sum of an entropy regularizer and a linear function. It is known that the state-of-the-art  solvers for this problem, which are based on Sinkhorn's method (also known as RSA or balancing method), can fail to work, when the entropy-regularization parameter is small. We consider the above optimization problem as a particular instance of a general strongly convex optimization problem with linear constraints. We propose a new algorithm to solve this general class of problems. Our approach is based on the transition to the dual problem. First, we introduce a new accelerated gradient method with adaptive choice of gradient's Lipschitz constant. Then, we apply this method to the dual problem and show, how to reconstruct an approximate solution to the primal problem with provable convergence rate. We prove the rate $O(1/k^2)$, $k$ being the iteration counter, both for the absolute value of the primal objective residual and constraints infeasibility. Our method has similar to Sinkhorn's method complexity of each iteration, but is faster and more stable numerically, when the regularization parameter is small. We illustrate the advantage of our method by numerical experiments for the two mentioned applications. We show that there exists a threshold, such that, when the regularization parameter is smaller than this threshold, our method outperforms the Sinkhorn's method in terms of computation time.
\end{abstract}


\section*{Introduction}
\label{S:Intro}
The main problem, we consider, is convex optimization problem of the following form
\begin{equation}
(P_1) \quad \quad \min_{x\in Q \subseteq E} \left\{ f(x) : A_1x =b_1, A_2x - b_2 \in -K \right\},
\notag
\end{equation}
where $E$ is a finite-dimensional real vector space, $Q$ is a simple closed convex set, $A_1$, $A_2$ are given linear operators from $E$ to some finite-dimensional real vector spaces $H_1$ and $H_2$ respectively, $b_1 \in H_1$, $b_2 \in H_2$ are given, $K \subseteq H_2$ is some cone, $f(x)$ is a $\gamma$-strongly convex function  on $Q$ with respect to some chosen norm $\|\cdot\|_E$ on $E$. The last means that, for any $x,y \in Q$, $f(y) \geq f(x) + \la \nabla f(x) , y-x \ra + \frac{\gamma}{2}\|x-y\|^2_E$, where $\nabla f(x)$ is any subgradient of $f(x)$ at $x$ and hence is an element of the dual space $E^*$. Also we denote the value of a linear function $\lambda \in E^*$ at $x\in E$ by $\la \lambda, x \ra$.

We are motivated to consider the described class of problems by two particular applications. The first one comes from transportation research and consists in recovering a matrix of traffic demands between city districts from the information on population and workplace capacities of each district. As it is shown in \cite{gasnikov2016evolution}, a natural model of districts' population dynamics leads to an entropy-linear programming optimization problem for the traffic demand matrix estimation. In this case, the objective function in $(P_1)$ is a sum of an entropy function and a linear function, see the formal problem statement in Subsection \ref{S:Exmpl}. It is important to note also that the entropy function is multiplied by a regularization parameter $\gamma$ and the model is close to reality, when the regularization parameter is small. The same approach is used in IP traffic matrix estimation \cite{zhang2005estimating}.

The second application is the calculation of regularized optimal transport (ROT) between two probability measures introduced in \cite{cuturi2013sinkhorn}. The idea is to regularize the objective function in the classical optimal transport linear programming problem \cite{kantorovich1942translocation} by entropy of the transportation plan, see the formal problem statement in Subsection \ref{S:Exmpl}. This leads to the same type of problem with a regularization parameter as in the traffic demands matrix estimation. 
As it is argued in \cite{cuturi2016smoothed}, for the case of discretization of continuous probability measures, entropy regularization allows to obtain a better approximation for the optimal transportation plan than the solution of the original linear programming problem. A the same time, the regularization parameter $\gamma$ should be small. Otherwise, the solution of the regularized optimal transport problem will be a bad approximation for the original optimal transport problem.
To sum up, in both applications, it is important to solve regularized problems with \emph{small regularization parameter}.

The problem statement $(P_1)$ covers many other applications besides mentioned above. For example, general entropy-linear programming (ELP) problem \cite{fang1997entropy} arises in econometrics \cite{golan1996maximum}, modeling in science and engineering \cite{kapur1989maximum}. 
Such machine learning approaches as ridge regression \cite{hastie2001elements} and elastic net \cite{zou2005regularization} lead to the same type of problem.

\subsection*{Related Work}

\textbf{Sinkhorn's, RSA or balancing type methods.}
Special types of Problem $(P_1)$, such as traffic matrix estimation and regularized optimal transport, have efficient matrix-scaling-based solvers such as balancing algorithm, \cite{bregman1967proof}, Sinkhorn's method, \cite{sinkhorn1974diagonal,cuturi2013sinkhorn}, RAS algorithm \cite{kalantari1993rate}. Strong points of these algorithms are fast convergence in practice and easy parallel implementation. At the same time, these algorithms are suitable only for Problem $(P_1)$ with special type of linear equality constraints. A generalization for a problem with a special type of linear inequalities constraints was suggested in \cite{benamou2015iterative}, but without convergence rate estimates. Recently, \cite{chizat2016scaling} extended the approach of \cite{cuturi2013sinkhorn} for other special classes of entropy-minimization problems. 

The problem of instability of the matrix-scaling approach for problems with small regularization parameter was addressed in \cite{schmitzer2016stabilized}, but the proposed techniques are less suitable for parallel computations than the initial algorithm. Besides instability issue, the theoretical analysis of this approach is insufficient. There is a proof of linear convergence of the Sinkhorn's method \cite{franklin1989scaling}, but the theoretical bound is much worse than the rate in practice and theoretical rate is obtained in terms of convergence in a special metric, which is hard to interpret. 
Another weak point of matrix-scaling-based approach is essential dual nature of the algorithms, which makes it hard to obtain convergence rate for the primal variable and control the accuracy of the obtained approximation for the optimal transport plan or traffic matrix. 

An alternative matrix scaling algorithm was proposed in \cite{allen2017much} together with theoretical analysis, but this method seems to be hard to implement in practice and no experimental results were reported. A Sinkhorn's-algorithm-based approach to solve the linear-programming optimal transport problem is suggested in \cite{altschuler2017near-linear} with promising theoretical bounds and practical implementability. But their approach requires to take small regularization parameter for the bounds to hold, making Sinkhorn's algorithm unstable. 

In any case, all the mentioned algorithms are designed for a special instance of Problem $(P_1)$.

\textbf{First-order methods for constrained problems.}
We consider Problem $(P_1)$ in large-scale setting, when the natural choice is some first-order method. Due to the presence of linear constraints, the applicability of projected-gradient-type methods to the primal problem is limited. Thus, the most common approach involves construction of the dual problem and primal-dual updates during the algorithm progress. There are many algorithms of this type like ADMM \cite{boyd2011distributed,goldstein2014fast} and other primal-dual methods \cite{chambolle2011first-order,beck2014fast}, see the extensive review in \cite{tran-dinh2014constrained}. As it is pointed in \cite{tran-dinh2014constrained}, these methods have the following drawbacks. They assume that the proximal operator for the function $f$ is available and make some additional assumptions. These methods don't have appropriate convergence rate characterization: if any, the rates are non-optimal and are either only for the dual problem or for some weighted sum of primal objective residual and linear constraints infeasibility. \cite{tran-dinh2014constrained} themselves develop a good alternative, based only on the assumption of proximal tractability of the function $f$, but only for problems with linear equality constraints. This approach was further developed in \cite{yurtsever2015universal} for more general types of constraints. The key feature of the algorithm developed there is its adaptivity to the unknown level of smoothness in the dual problem. Nevertheless, the provided stopping criterion, which is based on the prescribed number of iterations, requires to know all the smoothness parameters. \cite{tran-dinh2015smooth} propose algorithms with optimal rates of convergence for a more general class of problems, but, for the case of strongly convex $f$, they assume that it is strongly convex with respect to a Euclidean-type norm. Thus, their approach is not applicable to entropy minimization problems, which are our main focus. 

An advanced ADMM with provable convergence rate with appropriate convergence characterization was proposed in \cite{ouyang2015accelerated}, but only for the case of equality constraints and Lipschitz-smooth $f$, which does not cover the case of entropy minimization. A general primal-dual framework for unconstrained problems was proposed in \cite{dunner2016primal-dual}, but it is not applicable in our setting. An adaptive to unknown Lipschitz constant algorithm for primal-dual problems was developed in \cite{malitsky2016first-order}, but the authors work with a different from our problem statement and the case of strongly convex objective is considered only in Euclidean setting, which also does not cover the case of entropy minimization.

Several recent algorithms \cite{patrascu2015rate,gasnikov2016efficient,chernov2016fast,li2016inexact} are based on the application of accelerated gradient method \cite{nesterov2004introduction,nesterov2005smooth} to the dual problem and have optimal rates. At the same time, these works do not consider general types of constraints as in Problem $(P_1)$. Also the proposed algorithms use, as an input parameter, an estimate of the Lipschitz constant of the gradient in the dual problem, which can be very pessimistic and lead to slow convergence. 
%
%

\subsection*{Our approach and contributions	}
Our approach is based on the transition to the dual problem for $(P_1)$. Since $f$ is strongly convex, the objective in the dual problem has Lipschitz-continuous gradient and the Lipschitz constant can be estimated. The negative points of this approach are that the estimate for the Lipschitz constant can be very pessimistic and that the feasible set of the dual problem is unbounded. 

We develop a new accelerated gradient method which is interesting by itself. This method uses line-search idea of \cite{nesterov2006cubic} to adapt to the Lipschitz constant of the objective's gradient and, hence, can use the local smoothness information to make larger steps compared to the standard accelerated gradient method \cite{nesterov2004introduction}. Also our method uses general proximal setup and, thus, can be adopted to the geometry of the problem at hand. Another good point is that it uses only one proximal step as opposed to, for, example \cite{nesterov2005smooth}.

We apply our method to the dual problem and supply it with a procedure to reconstruct the approximate solution of the primal problem. Despite the unboundedness of the dual feasible set, we prove convergence rates for primal objective residual, dual objective residual and constraints infeasibility.

Finally, in the experiments, we show that our algorithm is a stable alternative to matrix-scaling approach for solving regularized optimal transport problems and traffic matrix estimation problems. At the same time, our algorithm has the same complexity of each iteration and uses only matrix-vector multiplication and vector summation, which is amenable for parallel computations.
To sum up, our contributions in this paper are as follows.
\begin{enumerate}
	\item We propose a new accelerated gradient method with new analysis, which, in contrast to \cite{nesterov2005smooth} uses only one proximal mapping on each step and is adaptive to the Lipschitz constant of the objective's gradient, and, in contrast to \cite{gasnikov2016universal} does not accumulate the history of the gradients.
	\item In contrast to the existing methods for constrained problems in \cite{boyd2011distributed,goldstein2014fast,chambolle2011first-order,beck2014fast,tran-dinh2014constrained,yurtsever2015universal,tran-dinh2015smooth,malitsky2016first-order,patrascu2015rate,gasnikov2016efficient,chernov2016fast,li2016inexact}, we propose an algorithm for Problem $(P_1)$ with general linear equality and cone constraints; with optimal rate of convergence in terms of both primal objective residual and constraints infeasibility; with adaptivity to the Lipschitz constant of the objective's gradient; with online stopping criterion, which does not require the knowledge of this Lipschitz constant; with ability to work with entropy function as $f$.	
	\item In contrast to existing algorithms for solving entropy-regularized optimal transport problems \cite{bregman1967proof,sinkhorn1974diagonal,cuturi2013sinkhorn,kalantari1993rate,benamou2015iterative,schmitzer2016stabilized,allen2017much,altschuler2017near-linear}, we provide an algorithm with provable convergence rate, which can be easily implemented in practice and is more stable, when the regularization parameter is small. 
	\item In the experiments, we show that our algorithm is better than the Sinkhorn's method in situations of small regularization parameter in the primal problem, which means that the dual problem becomes less smooth problem.
\end{enumerate}

The rest of the paper is organized as follows. 
In Section \ref{S:prel}, we introduce notation, definition of approximate solution to Problem $(P_1)$, main assumptions, and particular examples of $(P_1)$ in applications. In Section \ref{S:ASTM}, we introduce our new accelerated gradient method for general convex problems and provide its convergence rate analysis. Section \ref{S:ASTM} is devoted to primal-dual algorithm for Problem $(P_1)$ and its convergence analysis. Finally, in Section \ref{S:Num}, we present the results of the numerical experiments for regularized optimal transport and traffic matrix estimation problems.

\section{Preliminaries}
\label{S:prel}

\subsection{Notation}
For any finite-dimensional real vector space $E$, we denote by $E^*$ its dual. We denote the value of a linear function $\lambda \in E^*$ at $x\in E$ by $\la \lambda, x \ra$. Let $\|\cdot\|_E$ denote some norm on $E$ and $\|\cdot\|_{E,*}$ denote the norm on $E^*$ which is dual to $\|\cdot\|_E$
$$
\|\lambda\|_{E,*} = \max_{\|x\|_E \leq 1} \la \lambda, x \ra.
$$
In the special case, when $E$ is a Euclidean space, we denote the standard Euclidean norm by $\|\cdot\|_2$. Note that, in this case, the dual norm is also Euclidean.
For a cone $K \subseteq E$, the dual cone $K^* \subseteq E^*$ is defined as $K^*:=\{ \lambda \in E^*: \la \lambda, x \ra \geq 0 \quad \forall x \in K\}$. 
By $\partial f(x)$ we denote the subdifferential of a function $f(x)$ at a point $x$. Let $E_1, E_2$ be two finite-dimensional real vector spaces. For a linear operator $A:E_1 \to E_2$, we define its norm as follows
$$
\|A\|_{E_1 \to E_2} = \max_{x \in E_1,u \in E_2^*} \{\la u, A x \ra : \|x\|_{E_1} = 1, \|u\|_{E_2,*} = 1 \}.
$$
For a linear operator $A:E_1 \to E_2$, we define the adjoint operator $A^T: E_2^* \to E_1^*$ in the following way
$$
\la u, A x \ra = \la A^T u, x \ra, \quad \forall u \in E_2^*, \quad x \in E_1.
$$
We say that a function $f: E \to \R$ has a $L$-Lipschitz-continuous gradient if it is differentiable and its gradient satisfies Lipschitz condition
$$
\|\nabla f(x) - \nabla f(y) \|_{E,*} \leq L \|x-y\|_E, \quad \forall x,y \in E.
$$
Note that, from this inequality, it follows that
\begin{equation}
f(y) \leq f(x) + \la \nabla f(x) , y-x \ra + \frac{L}{2} \|x-y\|_E^2, \quad \forall x,y \in E.
\label{eq:nfLipDef}
\end{equation}
Also, for any $t \in \R$, we denote by  $\lceil t \rceil$ the smallest integer greater than or equal to $t$.

We characterize the quality of an approximate solution to Problem $(P_1)$ by three quantities $\e_f,\e_{eq}, \e_{in} > 0$.
\begin{Def}
We say that a point $\hat{x}$ is an $(\e_f,\e_{eq}, \e_{in})$-solution to Problem $(P_1)$ iff the following inequalities hold 
\begin{equation}
|f(\hat{x}) - Opt[P_1]| \leq \e_f , \quad \|A_1\hat{x}-b_1\|_2 \leq \e_{eq}, \quad \rho(A_2\hat{x}-b_2,-K) \leq \e_{in}.
\label{eq:sol_def}
\end{equation}
Here $Opt[P_1]$ denotes the optimal function value for Problem $(P_1)$, $\rho(A_2\hat{x}-b_2,-K):=\max_{\lambda^{(2)} \in K^*, \|\lambda^{(2)}\|_2 \leq 1} \la \lambda^{(2)}, A_2\hat{x}_{k+1}-b_2 \ra$. 
\end{Def}
Note that the last inequality in \eqref{eq:sol_def} is a natural generalization of linear constraints infeasibility measure $\|(A_2x_k-b_2)_+\|_{2}$ for the case $K = \R_+^n$. Here the vector $v_+$ denotes the vector with components $[v_+]_i=(v_i)_+=\max\{v_i,0\}$.

\subsection{Dual Problem}

The Lagrange dual problem to Problem $(P_1)$ is
\begin{equation}
(D_1) \quad \quad \max_{\lambda \in \Lambda} \left\{ - \la \lambda^{(1)}, b_1 \ra - \la \lambda^{(2)}, b_2 \ra + \min_{x\in Q} \left( f(x) + \la A_1^T \lambda^{(1)} + A_2^T \lambda^{(2)} ,x \ra \right) \right\}.
\notag
\end{equation}
Here we denote $\Lambda =\{\lambda = (\lambda^{(1)},\lambda^{(2)})^T \in H_1^* \times H_2^*: \lambda^{(2)} \in K^*\}$. 
It is convenient to rewrite Problem $(D_1)$ in the equivalent form of a minimization problem
\begin{align}
& (P_2) \quad \min_{\lambda \in \Lambda} \left\{   \la \lambda^{(1)}, b_1 \ra + \la \lambda^{(2)}, b_2 \ra + \max_{x\in Q} \left( -f(x) - \la A_1^T \lambda^{(1)} + A_2^T \lambda^{(2)} ,x \ra \right) \right\}. \notag
\end{align}
It is obvious that
\begin{equation}
Opt[D_1]=-Opt[P_2],
\label{eq:D_1_P_2_sol}
\end{equation}
where $Opt[D_1]$, $Opt[P_2]$ are the optimal function value in Problem $(D_1)$ and Problem $(P_2)$ respectively. 
The following inequality follows from the weak duality
\begin{equation}
Opt[P_1] \geq Opt[D_1].
\label{eq:wD}
\end{equation}
We denote
\begin{equation}
\vp(\lambda) = \vp(\lambda^{(1)}, \lambda^{(2)}) = \la \lambda^{(1)}, b_1 \ra + \la \lambda^{(2)}, b_2 \ra + \max_{x\in Q} \left( -f(x) - \la A_1^T \lambda^{(1)} + A_2^T \lambda^{(2)} ,x \ra \right).
\label{eq:vp_def}
\end{equation}
Since $f$ is strongly convex, $\vp(\lambda)$ is a smooth function and its gradient is equal to (see e.g. \cite{nesterov2005smooth})
\begin{equation}
\nabla \vp(\lambda) = \left(
\begin{aligned}
b_1 - A_1 x (\lambda)\\
b_2 - A_2 x (\lambda)\\
\end{aligned}
 \right),
\label{eq:nvp}
\end{equation}
where $x (\lambda)$ is the unique solution of the strongly-convex problem
\begin{equation}
\max_{x\in Q} \left( -f(x) - \la A_1^T \lambda^{(1)} + A_2^T \lambda^{(2)} ,x \ra \right).
\label{eq:inner}
\end{equation}
Note that $\nabla \vp(\lambda)$ is Lipschitz-continuous (see e.g. \cite{nesterov2005smooth}) with constant
$$
L \leq \frac{1}{\gamma}\left(\|A_1\|_{E \to H_1}^2+ \|A_2\|_{E \to H_2}^2\right).
$$
Previous works \cite{patrascu2015rate,gasnikov2016efficient,chernov2016fast,li2016inexact} rely on this quantity in the algorithm and use it to define the stepsize of the proposed algorithm. The drawback of this approach is that the above bound for the Lipschitz constant can be way too pessimistic. In this work, we propose an adaptive method, which has the same complexity bound, but is faster in practice due to the use of a "`local"' estimate for $L$ in the stepsize definition.


\subsection{Main Assumptions}
\label{S:main_assum}
We make the following main assumptions
\begin{enumerate}
	\item Function $f$ is $\gamma$-strongly convex.
	\item The problem \eqref{eq:inner} is simple in the sense that, for any $x \in Q$, it has a closed form solution or can be solved very fast up to the machine precision.
	\item The dual problem $(D_1)$ has a solution $\lambda^*=(\lambda^{*(1)},\lambda^{*(2)})^T$ and there exist some $R_1, R_2 >0$ such that
	\begin{equation}
	\|\lambda^{*(1)}\|_{2} \leq R_1 < +\infty, \quad \|\lambda^{*(2)}\|_{2} \leq R_2 < +\infty. 
	\label{eq:l_bound}
	\end{equation}	 
\end{enumerate}
It is worth noting that the quantities $R_1, R_2$ will be used only in the convergence analysis, but not in the algorithm itself.

\subsection{Examples of Problem $(P_1)$}
\label{S:Exmpl}
In this subsection, we describe several particular problems which can be written in the form of Problem $(P_1)$.

\textbf{Traffic demand matrix estimation, \cite{wilson2011entropy}, and Regularized optimal transport problem, \cite{cuturi2013sinkhorn}.}
\begin{equation}
\min_{X \in \R^{p \times p}_+} \left\{ \gamma \sum_{i,j=1}^p x_{ij} \ln x_{ij} + \sum_{i,j=1}^p c_{ij} x_{ij}: Xe=\mu, X^Te=\nu \right\},
\label{eq:ROT}
\end{equation}
where $e \in \R^p$ is the vector of all ones, $\mu, \nu \in S_p(1):= \{x \in \R^p : \sum_{i=1}^p x_i = 1, x_i \geq 0 , i=1,...,p \}$, $c_{ij} \geq 0, i,j=1,...,p$ are given, $\gamma > 0$ is the regularization parameter, $X^T$ is the transpose matrix of $X$, $x_{ij}$ is the element of the matrix $X$ in the $i$-th row and the $j$-th column. This problem with small value of $\gamma$ is our primary focus in this paper. 

\textbf{General entropy-linear programming problem, \cite{fang1997entropy}.}

\begin{equation}
\min_{x \in S_n(1)} \left\{ \sum_{i=1}^n x_i \ln \left(x_i/\xi_i\right)  : Ax =b \right\}
\notag
\end{equation}
for some given $\xi \in \R^n_{++} = \{x \in \R^n: x_i > 0 , i=1,...,n \}$.  


\section{Adaptive Similar Triangles Method}
\label{S:ASTM}
In this section, we consider a general optimization problem
\begin{equation}
\min_{\lambda \in \Lambda} \vp(\lambda),
\label{eq:prStGen}
\end{equation}
where $\Lambda$ is a closed convex, generally speaking, unbounded, set, $\vp(\lambda)$ is a convex function with $L$-Lipschitz-continuous gradient.

\subsection{Setting}
In this subsection, we introduce {\it proximal setup}, which is usually used in proximal gradient methods, see e.g. \cite{ben-tal2015lectures}.
We choose some norm $\|\cdot\|$ on the space of vectors $\lambda$ and a {\it prox-function} $d(\lambda)$ which is continuous, convex on $\Lambda$ and
\begin{enumerate}
	\item admits a continuous in $\lambda \in \Lambda^0$ selection of subgradients 	$\nabla d(\lambda)$, where $\lambda \in \Lambda^0 \subseteq \Lambda$  is the set of all $\lambda$, where $\nabla d(\lambda)$ exists;
	\item is $1$-strongly convex on $\Lambda$ with respect to $\|\cdot\|$, i.e., for any $\lambda \in \Lambda^0, \eta \in \Lambda$, $d(\eta)-d(\lambda) -\la \nabla d(\lambda) ,\eta-\lambda \ra \geq \frac12\|\eta-\lambda\|^2$.
\end{enumerate} 
We define also the corresponding {\it Bregman divergence} $V[\zeta] (\lambda) := d(\lambda) - d(\zeta) - \la \nabla d(\zeta), \lambda - \zeta \ra$, $\lambda \in \Lambda, \zeta \in \Lambda^0$. 
It is easy to see that
\begin{equation}
V[\zeta] (\lambda) \geq \frac12\|\lambda - \zeta\|^2, \quad \lambda \in \Lambda, \zeta \in \Lambda^0.
\label{eq:BFLowBound}
\end{equation}
Standard proximal setups, i.e. Euclidean, entropy, $\ell_1/\ell_2$, simplex,  nuclear norm, spectahedron can be found in \cite{ben-tal2015lectures}.

\subsection{Algorithm and Complexity Analysis}
In this subsection we present Adaptive Similar Triangles Method (ASTM) (see Algorithm \ref{Alg:ASTM} below). The name of the method is motivated by two factors. Firstly, it adaptively chooses the stepsize using local estimate $M_k$ of the Lipschitz constant $L$ of the gradient. Secondly, the choice of the point $\eta_{k+1}$ is such that the triangle $(\eta_k,\zeta_k,\zeta_{k+1})$ is similar to the triangle $(\eta_k,\lambda_{k+1},\eta_{k+1})$. 
\begin{algorithm}[h!]
\caption{Adaptive Similar Triangles Method (ASTM)}
\label{Alg:ASTM}
{\small
\begin{algorithmic}[1]
		\REQUIRE starting point $\lambda_0 \in \Lambda^0$, initial guess $L_0 >0$, prox-setup: $d(\lambda)$ -- $1$-strongly convex w.r.t. $\|\cdot\|$, $V[\zeta] (\lambda) := d(\lambda) - d(\zeta) - \la \nabla d(\zeta), \lambda - \zeta \ra$, $\lambda \in \Lambda, \zeta \in \Lambda^0$.
		\STATE Set $k=0$, $C_0=\alpha_0=0$, $\eta_0=\zeta_0=\lambda_0$.
		\REPEAT
			\STATE Set $M_k=L_k/2$.
			\REPEAT
				\STATE Set $M_k=2M_k$, find $\alpha_{k+1}$ as the largest root of the equation
				\begin{equation}
				C_{k+1}:=C_k+\alpha_{k+1} = M_k\alpha_{k+1}^2.
				\label{eq:alpQuadEq}
				\end{equation}
				\STATE 
				\begin{equation}
				\lambda_{k+1} = \frac{\alpha_{k+1}\zeta_k + C_k \eta_k}{C_{k+1}} .
				\label{eq:lambdakp1Def}
				\end{equation}
				\STATE 
				\begin{equation}
				\zeta_{k+1}=\arg \min_{\lambda \in \Lambda} \{V[\zeta_{k}](\lambda) + \alpha_{k+1}(\vp(\lambda_{k+1}) + \langle \nabla \vp(\lambda_{k+1}), \lambda - \lambda_{k+1} \rangle) \}.
				\label{eq:zetakp1Def}
				\end{equation}
				\STATE 
				\begin{equation}
				\eta_{k+1} = \frac{\alpha_{k+1}\zeta_{k+1} + C_k \eta_k}{C_{k+1}}.
				\label{eq:etakp1Def}
				\end{equation}
			\UNTIL{
			\begin{equation}
			\vp(\eta_{k+1}) \leq \vp(\lambda_{k+1}) + \la \nabla \vp(\lambda_{k+1}) ,\eta_{k+1} - \lambda_{k+1} \ra  +\frac{M_k}{2}\|\eta_{k+1} - \lambda_{k+1}\|^2.
			\label{eq:lipConstCheck}
			\end{equation}}
			\STATE Set $L_{k+1}=M_k/2$, $k=k+1$.
		\UNTIL{Option 1: $k = k_{max}$. \\ 
					Option 2: $R^2/C_k \leq \e$. \\
					Option 3:  
					$$
					\vp(\eta_k) - \min_{\lambda \in \Lambda: V[\zeta_0](\lambda) \leq R^2} \left\{ \sum_{i=0}^k \frac{\alpha_{i}}{C_k} \left( \vp(\lambda_{i}) + \la \nabla \vp(\lambda_{i}), \lambda - \lambda_{i}\ra \right) \right\} \leq \e.
					$$
					Here $R$ is such that $ V[\zeta_0](\lambda_*) \leq R^2$ and $\e$ is the desired accuracy.
		}
		\ENSURE The point $\eta_{k+1}$.	
\end{algorithmic}
}
\end{algorithm}

\begin{Lm}
	Algorithm \ref{Alg:ASTM} is defined correctly in the sense that the inner cycle of checking the inequality \eqref{eq:lipConstCheck} is finite.
	\label{Lm:0}
\end{Lm}

\begin{proof}
Since, before each check of the inequality \eqref{eq:lipConstCheck} on the step $k$, we multiply $M_k$ by 2, after finite number of these multiplications, we will have $M_k \geq L$. Since $\vp$ has $L$-Lipschitz-continuous gradient, due to \eqref{eq:nfLipDef}, we obtain that \eqref{eq:lipConstCheck} holds after finite number of these repetitions.
\end{proof}

\begin{Lm}
	Let the sequences $\{\lambda_k, \eta_k, \zeta_k, \alpha_k, C_k \}$, $k\geq 0$ be generated by Algorithm \ref{Alg:ASTM}. Then,  for all $\lambda \in \Lambda$, it holds that
	\begin{equation}
	\alpha_{k+1}\langle \nabla \vp(\lambda_{k+1}), \zeta_k - \lambda\rangle \leq C_{k+1}(\vp(\lambda_{k+1}) - \vp(\eta_{k+1})) + V[\zeta_k](\lambda) - V[\zeta_{k+1}](\lambda). 
	\label{eq:Lm1}
	\end{equation}
	\label{Lm:1}
\end{Lm}
\begin{proof}
Note that, from the optimality condition in \eqref{eq:zetakp1Def}, for any $\lambda \in \Lambda$, we have
\begin{equation}
\la \nabla V[\zeta_{k}](\zeta_{k+1}) + \alpha_{k+1} \nabla \vp(\lambda_{k+1}), \lambda - \zeta_{k+1} \ra \geq 0.
\label{eq:Lm1Pr1}
\end{equation}
By the definition of $V[\zeta](\lambda)$, we obtain, for any $\lambda \in \Lambda$,
\begin{align}
V[\zeta_k](\lambda) - V[\zeta_{k+1}](\lambda) - V[\zeta_k](\zeta_{k+1})  = &d(\lambda) - d(\zeta_k) - \la \nabla d(\zeta_k), \lambda-\zeta_k \ra \notag \\
 &- \left( d(\lambda) - d(\zeta_{k+1}) - \la \nabla d(\zeta_{k+1}), \lambda-\zeta_{k+1} \ra  \right) \notag \\
 &- \left( d(\zeta_{k+1}) - d(\zeta_k) - \la \nabla d(\zeta_k), \zeta_{k+1}-\zeta_k \ra \right) \notag \\
 &\hspace{-1em}= \la \nabla d(\zeta_k) - \nabla d(\zeta_{k+1}) , \zeta_{k+1} - \lambda \ra \notag \\
 &\hspace{-1em}= \la - \nabla V[\zeta_{k}](\zeta_{k+1}) , \zeta_{k+1} - \lambda \ra. 
\label{eq:magicV}
\end{align}
Further, for any $\lambda \in \Lambda$,
\begin{align}
\alpha_{k+1}\la \nabla \vp(\lambda_{k+1}), \zeta_{k} - \lambda\ra 
		&= \alpha_{k+1}\la \nabla \vp(\lambda_{k+1}), \zeta_{k} - \zeta_{k+1}\ra + \alpha_{k+1}\la \nabla \vp(\lambda_{k+1}), \zeta_{k+1} - \lambda\ra \notag \\
    &\hspace{-7em}\stackrel{\eqref{eq:Lm1Pr1}}{\leq} \alpha_{k+1}\la \nabla \vp(\lambda_{k+1}), \zeta_{k} - \zeta_{k+1}\ra + \la -\nabla V[\zeta_{k}](\zeta_{k+1}) , \zeta_{k+1} - \lambda  \ra \notag \\
	  &\hspace{-7em}\stackrel{\eqref{eq:magicV}}{=} \alpha_{k+1}\la \nabla \vp(\lambda_{k+1}), \zeta_{k} - \zeta_{k+1}\ra + V[\zeta_k](\lambda) - V[\zeta_{k+1}](\lambda) - V[\zeta_k](\zeta_{k+1}) \notag \\
		&\hspace{-7em}\stackrel{\eqref{eq:BFLowBound}}{\leq} \alpha_{k+1}\la \nabla \vp(\lambda_{k+1}), \zeta_{k} - \zeta_{k+1}\ra + V[\zeta_k](\lambda) - V[\zeta_{k+1}](\lambda) - \frac12\|\zeta_k-\zeta_{k+1}\|^2 \notag \\\
		&\hspace{-8em}\stackrel{\eqref{eq:lambdakp1Def},\eqref{eq:etakp1Def}}{=} C_{k+1} \la \nabla \vp(\lambda_{k+1}), \lambda_{k+1} - \eta_{k+1}\ra + V[\zeta_k](\lambda) - V[\zeta_{k+1}](\lambda) - \frac{C_{k+1}^2}{2\alpha_{k+1}^2}\|\lambda_{k+1} - \eta_{k+1}\|^2 \notag \\
		&\hspace{-7em}\stackrel{\eqref{eq:alpQuadEq}}{=} C_{k+1}\left(\la \nabla \vp(\lambda_{k+1}), \lambda_{k+1} - \eta_{k+1}\ra - \frac{M_k}{2}\|\lambda_{k+1} - \eta_{k+1}\|^2  \right) + V[\zeta_k](\lambda) - V[\zeta_{k+1}](\lambda) \notag \\
		&\hspace{-7em}\stackrel{\eqref{eq:lipConstCheck}}{\leq} C_{k+1}\left(  \vp(\lambda_{k+1}) - \vp(\eta_{k+1})  \right)  + V[\zeta_k](\lambda) - V[\zeta_{k+1}](\lambda) . \notag
\end{align}
\end{proof}

\begin{Lm}
Let the sequences $\{\lambda_k, \eta_k, \zeta_k, \alpha_k, C_k \}$, $k\geq 0$ be generated by Algorithm \ref{Alg:ASTM}. Then,  for all $\lambda \in \Lambda$, it holds that
	\begin{equation}
	C_{k+1} \vp(\eta_{k+1}) - C_k \vp(\eta_{k}) \leq \alpha_{k+1} \left( \vp(\lambda_{k+1}) + \la \nabla \vp(\lambda_{k+1}), \lambda - \lambda_{k+1}\ra \right) + V[\zeta_k](\lambda) - V[\zeta_{k+1}](\lambda).
	\label{eq:Lm2}
	\end{equation}
	\label{Lm:2}
\end{Lm}

\begin{proof}
For any $\lambda \in \Lambda$,
\begin{align}
\alpha_{k+1}\la \nabla \vp(\lambda_{k+1}), \lambda_{k+1} - \lambda\ra &= \alpha_{k+1}\la \nabla \vp(\lambda_{k+1}), \lambda_{k+1} - \zeta_{k}\ra + \alpha_{k+1}\la \nabla \vp(\lambda_{k+1}), \zeta_{k} - \lambda\ra \notag \\
    &\hspace{-8em}\stackrel{\eqref{eq:alpQuadEq},\eqref{eq:lambdakp1Def}}{=} C_{k}\la \nabla \vp(\lambda_{k+1}), \eta_{k} - \lambda_{k+1}\ra + \alpha_{k+1}\la \nabla \vp(\lambda_{k+1}), \zeta_{k} - \lambda\ra \notag \\
		&\hspace{-8em}\stackrel{\text{conv-ty}}{\leq} C_{k}\left(\vp(\eta_{k}) - \vp(\lambda_{k+1}) \right) + \alpha_{k+1}\la \nabla \vp(\lambda_{k+1}), \zeta_{k} - \lambda\ra \notag \\
		&\hspace{-7em}\stackrel{\eqref{eq:Lm1}}{\leq} C_{k}\left(\vp(\eta_{k}) - \vp(\lambda_{k+1}) \right) + C_{k+1}\left(  \vp(\lambda_{k+1}) - \vp(\eta_{k+1})  \right)  + V[\zeta_k](\lambda) - V[\zeta_{k+1}](\lambda) \notag \\		
		&\hspace{-7em}= \alpha_{k+1} \vp(\lambda_{k+1}) + C_k \vp(\eta_{k}) - C_{k+1} \vp(\eta_{k+1}) + V[\zeta_k](\lambda) - V[\zeta_{k+1}](\lambda).
\end{align}
Rearranging terms, we obtain the statement of the Lemma.
\end{proof}

\begin{Th}
\label{Th:ASTMConv}
Let the sequences $\{\lambda_k, \eta_k, \zeta_k, \alpha_k, C_k \}$, $k\geq 0$ be generated by Algorithm \ref{Alg:ASTM}. Then,  for all $k \geq 0$, it holds that
	\begin{equation}
	C_k\vp(\eta_k) \leq \min_{\lambda \in \Lambda} \left\{ \sum_{i=0}^k \alpha_{i} \left( \vp(\lambda_{i}) + \la \nabla \vp(\lambda_{i}), \lambda - \lambda_{i}\ra \right) + V[\zeta_0](\lambda) \right\}.
	\label{eq:ASTMConv}
	\end{equation}
The number of oracle calls after the iteration $k \geq 0$ does not exceed
\begin{equation}
 4k + 4 +2\log_2\left(\frac{L}{L_0}\right),
\label{eq:ASTMOrC}
\end{equation}
where $L$ is the Lipschitz constant for the gradient of $\vp$. 
\end{Th}

\begin{proof}
Let us change the counter in Lemma \ref{Lm:1} from $k$ to $i$ and sum all the inequalities for $i=0,...,k-1$. Then, for any $\lambda \in \Lambda$,
\begin{equation}
C_k\vp(\eta_{k}) - C_0 \vp(\eta_{0}) \leq \sum_{i=0}^{k-1} \alpha_{i+1} \left( \vp(\lambda_{i+1}) + \la \nabla \vp(\lambda_{i+1}), \lambda - \lambda_{i+1}\ra \right) + V[\zeta_0](\lambda) - V[\zeta_{k}](\lambda).
\label{eq:Th1Pr1}
\end{equation}
Whence, since $C_0 = \alpha_0 = 0$ and $V[\zeta_{k}](\lambda) \geq 0$, 
\begin{equation}
C_k\vp(\eta_{k}) \leq \sum_{i=0}^{k} \alpha_{i} \left( \vp(\lambda_{i}) + \la \nabla \vp(\lambda_{i}), \lambda - \lambda_{i}\ra \right) + V[\zeta_0](\lambda), \quad \lambda \in \Lambda.
\label{eq:Th1Pr2}
\end{equation}
Taking in the right hand side the minimum in $\lambda \in \Lambda$, we obtain the first statement of the Theorem.

The second statement of the Theorem is proved in the same way as in \cite{nesterov2006cubic}, but we provide the proof for the reader's convenience. Let us again change the iteration counter in Algorithm \ref{Alg:ASTM} from $k$ to $i$.
Let $j_i \geq 1$ be the total number of checks of the inequality \eqref{eq:lipConstCheck} on the step $i \geq 0$. Then, $j_0 = 1+\log_2\frac{M_0}{L_0}$ and, for $i \geq 1$, $M_i = 2^{j_i-1}L_i = 2^{j_i-1}\frac{M_{i-1}}{2}$. Thus, $j_i = 2+ \log_2\frac{M_i}{M_{i-1}}$, $i \geq 1$. Further, by the same reasoning as in Lemma \ref{Lm:1}, we obtain that $M_i \leq 2L$, $i\geq 0$. Then, the total number of checks of the inequality \eqref{eq:lipConstCheck} is 
\begin{equation}
\sum_{i=0}^{k}j_i=1+\log_2\frac{M_0}{L_0} + \sum_{i=1}^{k}\left(2+ \log_2\frac{M_i}{M_{i-1}}\right) = 2k+1+\log_2\frac{M_{k}}{L_0} \leq 2k+2+\log_2\frac{L}{L_0}.
\notag
\end{equation}
At the same time, each check of the inequality \eqref{eq:lipConstCheck} requires two oracle calls. This proves the second statement of the Theorem.
\end{proof}	

\begin{Cor}
\label{Cor:ASTMConv2}
Let the sequences $\{\lambda_k, \eta_k, \zeta_k, \alpha_k, C_k \}$, $k\geq 0$ be generated by Algorithm \ref{Alg:ASTM}. Then,  for all $k \geq 0$, it holds that
	\begin{equation}
	\vp(\eta_k) - \min_{\lambda \in \Lambda} \vp(\lambda) \leq \frac{V[\zeta_0](\lambda_*)}{C_k} ,
	\label{eq:ASTMConv2}
	\end{equation}
	where $\lambda_*$ is the solution of $\min_{\lambda \in \Lambda} \vp(\lambda)$ s.t. $V[\zeta_0](\lambda_*)$ is minimal among all the solutions.
\end{Cor}

\begin{proof}
Let $\lambda_*$ be the solution of $\min_{\lambda \in \Lambda} \vp(\lambda)$ s.t. $V[\zeta_0](\lambda_*)$ is minimal among all the solutions. Using convexity of $\vp$, from Theorem 	\ref{Th:ASTMConv}, we obtain
$$
C_k\vp(\eta_k) \leq \sum_{i=0}^{k} \alpha_{i}   \vp(\lambda_*) + V[\zeta_0](\lambda_*).
$$
Since $C_k = \sum_{i=0}^{k} \alpha_{i}$, we obtain the statement of the Corollary.
\end{proof}

The following Corollary justifies the stopping criteria in Algorithm \ref{Alg:ASTM}.
\begin{Cor}
\label{Cor:ASTMStop}
Let $\lambda_*$ be a solution of $\min_{\lambda \in \Lambda} \vp(\lambda)$ such that $V[\zeta_0](\lambda_*)$ is minimal among all the solutions. Let $R$ be such that $ V[\zeta_0](\lambda_*) \leq R^2$ and $\e$ be the desired accuracy. Let the sequences $\{\lambda_k, \eta_k, \zeta_k, \alpha_k, C_k \}$, $k\geq 0$ be generated by Algorithm \ref{Alg:ASTM}. 
Then, if one of the following inequalities holds
\begin{align}
	&R^2/C_k \leq \e, \label{eq:ASTMStop1}\\
	&\vp(\eta_k) - \min_{\lambda \in \Lambda: V[\zeta_0](\lambda) \leq R^2} \left\{ \sum_{i=0}^k \frac{\alpha_{i}}{C_k} \left( \vp(\lambda_{i}) + \la \nabla \vp(\lambda_{i}), \lambda - \lambda_{i}\ra \right) \right\} \leq \e, \label{eq:ASTMStop2}
\end{align}
then 
	\begin{equation}
	\vp(\eta_k) - \min_{\lambda \in \Lambda} \vp(\lambda) \leq \e.
	\label{eq:ASTMConv3}
	\end{equation}	
\end{Cor}

\begin{proof}
If the inequality \eqref{eq:ASTMStop1} holds, the statement of the Corollary follows from inequality $ V[\zeta_0](\lambda_*) \leq R^2$ Corollary \ref{Cor:ASTMConv2}.

Since $ V[\zeta_0](\lambda_*) \leq R^2$, the point $\lambda_*$ is a feasible point in the problem 
$$
\min_{\lambda \in \Lambda: V[\zeta_0](\lambda) \leq R^2} \left\{ \sum_{i=0}^k \frac{\alpha_{i}}{C_k} \left( \vp(\lambda_{i}) + \la \nabla \vp(\lambda_{i}), \lambda - \lambda_{i}\ra \right) \right\}.
$$
Then, by convexity of $\vp$, we obtain
\begin{align}
\min_{\lambda \in \Lambda: V[\zeta_0](\lambda) \leq R^2} \left\{ \sum_{i=0}^k \frac{\alpha_{i}}{C_k} \left( \vp(\lambda_{i}) + \la \nabla \vp(\lambda_{i}), \lambda - \lambda_{i}\ra \right) \right\} & \leq \sum_{i=0}^k \frac{\alpha_{i}}{C_k} \left( \vp(\lambda_{i}) + \la \nabla \vp(\lambda_{i}), \lambda_* - \lambda_{i}\ra \right) \notag \\
& \leq \vp(\lambda_*). \notag
\end{align}
This and \eqref{eq:ASTMStop2} finishes the proof.
\end{proof}

Let us now obtain the lower bound for the sequence $C_k$, $k \geq 0$, which will give the rate of convergence for Algorithm \ref{Alg:ASTM}.
\begin{Lm}
\label{Lm:CkGrowth}
Let the sequence $\{C_k \}$, $k\geq 0$ be generated by Algorithm \ref{Alg:ASTM}. Then,  for all $k \geq 1$ it holds that
\begin{equation}
C_k \geq \frac{(k+1)^2}{8L},
\label{eq:CkGrowth}
\end{equation}	
where $L$ is the Lipschitz constant for the gradient of $\vp$.
\end{Lm}

\begin{proof}
As we mentioned in the proof of Theorem \ref{Th:ASTMConv}, $M_k \leq 2L$, $k\geq 0$. For $k=1$, since $\alpha_0 = 0 $ and $A_1 = \alpha_0+\alpha_1= \alpha_1$, we have from \eqref{eq:alpQuadEq} 
	\begin{equation*}
		C_1 = \alpha_1 = \frac{1}{M_1} \geq \frac{1}{2L}.
	\end{equation*}
	Hence, \eqref{eq:CkGrowth} holds for $k = 1$.
	
	Let us now assume that \eqref{eq:CkGrowth} holds for some $k \geq 1$ and prove that it holds for $k+1$.
	From \eqref{eq:alpQuadEq}  we have a quadratic equation for $\alpha_{k+1}$
	\begin{equation*}
	M_{k}\alpha_{k+1}^2 - \alpha_{k+1} - C_k = 0.
	\end{equation*}
	Since we need to take the largest root, we obtain,
	\begin{align}
	\alpha_{k+1} & = \frac{1 + \sqrt{\uprule 1 + 4M_{k}C_{k}}}{2M_{k}} = \frac{1}{2M_{k}} + \sqrt{\frac{1}{4M_{k}^2} + \frac{C_{k}}{M_{k}}} \geq 
	\frac{1}{2M_{k}} + \sqrt{\frac{C_{k}}{M_{k}}} \notag \\
	&\geq
	\frac{1}{4L} + \frac{1}{\sqrt{2L}}\frac{k+1}{2\sqrt{2L}} =
	\frac{k+2}{4L}, \notag
	\end{align}
	where we used the induction assumption that \eqref{eq:CkGrowth} holds for $k$.
	Using the obtained inequality, from \eqref{eq:alpQuadEq} and \eqref{eq:CkGrowth} for $k$, we get
	\begin{equation*}
	C_{k+1} = C_k + \alpha_{k+1} \geq \frac{(k+1)^2}{8L} + \frac{k+2}{4L} \geq \frac{(k+2)^2}{8L}.
	\end{equation*}
\end{proof}

\begin{Cor}
\label{Cor:ASTMConv3}
Let the sequences $\{\lambda_k, \eta_k, \zeta_k\}$, $k\geq 0$ be generated by Algorithm \ref{Alg:ASTM}. Then,  for all $k \geq 1$, it holds that
	\begin{equation}
	\vp(\eta_k) - \min_{\lambda \in \Lambda} \vp(\lambda) \leq \frac{8LV[\zeta_0](\lambda_*)}{(k+1)^2} ,
	\label{eq:ASTMConv2}
	\end{equation}
	where $\lambda_*$ is the solution of $\min_{\lambda \in \Lambda} \vp(\lambda)$ s.t. $V[\zeta_0](\lambda_*)$ is minimal among all the solutions.
\end{Cor}

	%
%
	%
	%
	%
%
%
%
%
%
%
%
%
%
%
%
%
 
\section{Solving the Dual Problem by ASTM and Reconstructing a Primal Problem Solution}
\label{S:ASTM}

In this section, we return to the primal-dual pair of problems $(P_1)$-$(D_1)$. We apply Algorithm \ref{Alg:ASTM} to the problem $(P_2)$ and incorporate in the algorithm a procedure, which allows to reconstruct also an approximate solution of the problem $(P_1)$. 

\subsection{Setup}
We choose Euclidean proximal setup, which means that we introduce euclidean norm $\|\cdot \|_2$ in the space of vectors $\lambda$ and choose the prox-function $d(\lambda) = \frac12\|\lambda\|_2^2$. Then, we have $V[\zeta](\lambda) = \frac12\|\lambda-\zeta\|_2^2$.

\subsection{New Algorithm and Primal-Dual Analysis}
Our primal-dual algorithm for Problem $(P_1)$ is listed below as Algorithm \ref{Alg:PDASTM}. Note that, in this case, the set $\Lambda$ has a special structure
$$
\Lambda =\{\lambda = (\lambda^{(1)},\lambda^{(2)})^T \in H_1^* \times H_2^*: \lambda^{(2)} \in K^*\}
$$
as well as $\vp(\lambda)$ and $\nabla \vp(\lambda)$ are defined in \eqref{eq:vp_def} and \eqref{eq:nvp} respectively. Thus, the step \eqref{eq:PDzetakp1Def} of the algorithm can be written explicitly.
\begin{align}
				&\zeta_{k+1}^{(1)} = \zeta_{k}^{(1)} + \frac{1}{\alpha_{k+1}} (A_1x(\lambda_{k+1})-b_1), \notag \\
				&\zeta_{k+1}^{(2)} = \Pi_{K^*} \left(\zeta_{k}^{(2)} + \frac{1}{\alpha_{k+1}} (A_2x(\lambda_{k+1})-b_2)  \right), \notag  
\end{align}
where $\Pi_{K^*}(\cdot)$ denotes euclidean projection on the cone $K^*$.

It is worth noting that, besides solution of the problem \eqref{eq:inner}, the algorithm uses only matrix-vector multiplications and vector operations, which made it amenable for parallel implementation.
\begin{algorithm}[h!]
\caption{Primal-Dual Adaptive Similar Triangles Method (PDASTM)}
\label{Alg:PDASTM}
{\small
\begin{algorithmic}[1]
   \REQUIRE starting point $\lambda_0 = 0$, initial guess $L_0 >0$, accuracy $\tilde{\e}_f,\tilde{\e}_{eq},\tilde{\e}_{in} > 0$.
   \STATE Set $k=0$, $C_0=\alpha_0=0$, $\eta_0=\zeta_0=\lambda_0=0$.
   \REPEAT
			\STATE Set $M_k=L_k/2$.
			\REPEAT
				\STATE Set $M_k=2M_k$, find $\alpha_{k+1}$ as the largest root of the equation
				\begin{equation}
				C_{k+1}:=C_k+\alpha_{k+1} = M_k\alpha_{k+1}^2.
				\label{eq:PDalpQuadEq}
				\end{equation}
				\STATE Calculate 
				\begin{equation}
				\lambda_{k+1} = (\lambda_{k+1}^{(1)},\lambda_{k+1}^{(2)})^T = \frac{\alpha_{k+1}\zeta_k + C_k \eta_k}{C_{k+1}} .
				\label{eq:PDlambdakp1Def}
				\end{equation}
				\STATE Calculate 
				\begin{align}
				\zeta_{k+1}&=(\zeta_{k+1}^{(1)},\zeta_{k+1}^{(2)})^T  \notag \\
				&= \arg \min_{\lambda \in \Lambda} \left\{\frac12\|\lambda-\zeta_{k}\|_2^2 + \alpha_{k+1}(\vp(\lambda_{k+1}) + \langle \nabla \vp(\lambda_{k+1}), \lambda - \lambda_{k+1} \rangle) \right\}.
				\label{eq:PDzetakp1Def}
				\end{align}
				\STATE Calculate
				\begin{equation}
				\eta_{k+1} = (\eta_{k+1}^{(1)},\eta_{k+1}^{(2)})^T = \frac{\alpha_{k+1}\zeta_{k+1} + C_k \eta_k}{C_{k+1}}.
				\label{eq:PDetakp1Def}
				\end{equation}
			\UNTIL{
				\begin{equation}
					\vp(\eta_{k+1}) \leq \vp(\lambda_{k+1}) + \la \nabla \vp(\lambda_{k+1}) ,\eta_{k+1} - \lambda_{k+1} \ra  +\frac{M_k}{2}\|\eta_{k+1} - \lambda_{k+1}\|_2^2.
					\label{eq:PDlipConstCheck}
				\end{equation}
			}
			\STATE Set
				\begin{equation}
					\hat{x}_{k+1} = \frac{1}{C_{k+1}}\sum_{i=0}^{k+1} \alpha_i x(\lambda_i) = \frac{\alpha_{k+1}x(\lambda_{k+1})+C_k\hat{x}_{k}}{C_{k+1}}.
				\notag
				\end{equation}
			\STATE Set $L_{k+1}=M_k/2$, $k=k+1$.
  \UNTIL{$|f(\hat{x}_{k+1})+\vp(\eta_{k+1})| \leq \tilde{\e}_f$, $\|A_1\hat{x}_{k+1}-b_1\|_{2} \leq \tilde{\e}_{eq}$, $\rho(A_2\hat{x}_{k+1}-b_2,-K) \leq \tilde{\e}_{in}$.}
	\ENSURE The points $\hat{x}_{k+1}$, $\eta_{k+1}$.	
\end{algorithmic}
}
\end{algorithm}

\begin{Th}
Let the assumptions listed in Subsection \ref{S:main_assum} hold. Then Algorithm \ref{Alg:PDASTM} will stop not later than $k$ equals to
$$
\max \left\{ \left\lceil \sqrt{\frac{16L(R_1^2+R_2^2)}{\tilde{\e}_f}}  \right\rceil, \left\lceil \sqrt{\frac{16L(R_1^2+R_2^2)}{R_1\tilde{\e}_{eq}}} \right\rceil, \left\lceil \sqrt{\frac{16L(R_1^2+R_2^2)}{R_2\tilde{\e}_{in}}} \right\rceil \right\}. 
$$
Moreover, not later than $k$ equals to  
$$
\max \left\{ \left\lceil \sqrt{\frac{32L(R_1^2+R_2^2)}{\e_f}}  \right\rceil, \left\lceil \sqrt{\frac{16L(R_1^2+R_2^2)}{R_1\e_{eq}}} \right\rceil, \left\lceil \sqrt{\frac{16L(R_1^2+R_2^2)}{R_2\e_{in}}} \right\rceil \right\}, 
$$
the point $\hat{x}_{k+1}$ generated by Algorithm \ref{Alg:PDASTM} is an approximate solution to Problem $(P_1)$ in the sense of \eqref{eq:sol_def}.
\label{Th:PDASTMConv}
\end{Th}

\begin{proof}
The proof mostly follows the steps of our previous work \cite{chernov2016fast}, but we give the proof for the reader's convenience. From Theorem \ref{Th:ASTMConv} with specific choice of the Bregman divergence, since $\zeta_0=0$, we have, for all $k\geq 0$,  
\begin{equation}
C_k \vp(\eta_k) \leq \min_{\lambda \in \Lambda} \left\{  \sum_{i=0}^k{\alpha_i \left( \vp(\lambda_i) + \la \nabla \vp(\lambda_i), \lambda-\lambda_i \ra \right) } + \frac{1}{2} \|\lambda\|_2^2 \right\}
\label{eq:FGM_compl}
\end{equation}
Let us introduce a set $\Lambda_R =\{\lambda = (\lambda^{(1)},\lambda^{(2)})^T: \lambda^{(2)} \in K^*, \|\lambda^{(1)}\|_2 \leq 2R_1, \|\lambda^{(2)}\|_2 \leq 2R_2 \}$ where $R_1$, $R_2$ are given in \eqref{eq:l_bound}. Then, from \eqref{eq:FGM_compl}, we obtain
\begin{align}
C_k \vp(\eta_k) &\leq \min_{\lambda \in \Lambda} \left\{  \sum_{i=0}^k{\alpha_i \left( \vp(\lambda_i) + \la \nabla \vp(\lambda_i), \lambda-\lambda_i \ra \right) } + \frac{1}{2} \|\lambda\|_2^2 \right\} \ \notag \\
&\leq \min_{\lambda \in \Lambda_R} \left\{  \sum_{i=0}^k{\alpha_i \left( \vp(\lambda_i) + \la \nabla \vp(\lambda_i), \lambda-\lambda_i \ra \right) } + \frac{1}{2} \|\lambda\|_2^2 \right\}  \notag \\
&\leq \min_{\lambda \in \Lambda_R} \left\{  \sum_{i=0}^k{\alpha_i \left( \vp(\lambda_i) + \la \nabla \vp(\lambda_i), \lambda-\lambda_i \ra \right) } \right\} + 2(R_1^2+R_2^2).
\label{eq:proof_st_1}
\end{align}
On the other hand, from the definition \eqref{eq:vp_def} of $\vp(\lambda)$, we have
\begin{align}
 \vp(\lambda_i) & = \vp(\lambda_i^{(1)},\lambda_i^{(2)}) = \la \lambda_i^{(1)}, b_1 \ra + \la \lambda_i^{(2)}, b_2 \ra \notag \\
& \hspace{1em} + \max_{x\in Q} \left( -f(x) - \la A_1^T \lambda_i^{(1)} + A_2^T \lambda_i^{(2)} ,x \ra \right) \notag \\
& = \la \lambda_i^{(1)}, b_1 \ra + \la \lambda_i^{(2)}, b_2 \ra - f(x(\lambda_i)) - \la A_1^T \lambda_i^{(1)} + A_2^T \lambda_i^{(2)} ,x(\lambda_i) \ra . \notag
\end{align}
Combining this equality with \eqref{eq:nvp}, we obtain
\begin{align}
\vp(\lambda_i) - \la \nabla \vp (\lambda_i), \lambda_i \ra &= \vp(\lambda_i^{(1)},\lambda_i^{(2)}) - \la \nabla \vp(\lambda_i^{(1)},\lambda_i^{(2)}), (\lambda_i^{(1)},\lambda_i^{(2)})^T \ra = \notag \\
& = \la \lambda_i^{(1)}, b_1 \ra + \la \lambda_i^{(2)}, b_2 \ra - f(x(\lambda_i)) - \la A_1^T \lambda_i^{(1)} + A_2^T \lambda_i^{(2)} ,x(\lambda_i) \ra \notag \\
& \hspace{1em} - \la b_1-A_1 x(\lambda_i),\lambda_i^{(1)} \ra - \la b_2-A_2 x(\lambda_i),\lambda_i^{(2)} \ra = - f(x(\lambda_i)). \notag
\end{align}
Summing these inequalities from $i=0$ to $i=k$ with the weights $\{\alpha_i\}_{i=1,...k}$, we get, using the convexity of $f$
\begin{align}
&  \sum_{i=0}^k{\alpha_i \left( \vp(\lambda_i) + \la \nabla \vp(\lambda_i), \lambda-\lambda_i \ra \right) }  \notag \\
& = -\sum_{i=0}^k \alpha_i f(x(\lambda_i)) + \sum_{i=0}^k \alpha_i \la (b_1-A_1 x(\lambda_i),b_2-A_2 x(\lambda_i))^T, (\lambda^{(1)},\lambda^{(2)})^T \ra  \notag \\
& \leq -C_kf(\hat{x}_k) + C_k \la (b_1-A_1 \hat{x}_k,b_2-A_2 \hat{x}_k)^T, (\lambda^{(1)},\lambda^{(2)})^T \ra . \notag
\end{align}
Substituting this inequality to \eqref{eq:proof_st_1}, we obtain
\begin{align}
C_k \vp(\eta_k)  \leq &-C_kf(\hat{x}_k)  \notag \\
& + C_k \min_{\lambda \in \Lambda_R} \left\{  \la (b_1-A_1 \hat{x}_k,b_2-A_2 \hat{x}_k)^T, (\lambda^{(1)},\lambda^{(2)})^T \ra \right\} + 2(R_1^2+R_2^2) . \notag
\end{align}
Finally, since 
\begin{align}
&\max_{\lambda \in \Lambda_R} \left\{  \la (-b_1+A_1 \hat{x}_k,-b_2+A_2 \hat{x}_k)^T, (\lambda^{(1)},\lambda^{(2)})^T \ra \right\}  \notag \\
&\hspace{3em} =2 R_1 \|A_1 \hat{x}_k - b_1 \|_2 + 2R_2 \rho(A_2\hat{x}_{k}-b_2,-K), \notag
\end{align} 
we obtain
\begin{equation}
\vp(\eta_k) + f(\hat{x}_k) +2 R_1 \|A_1 \hat{x}_k - b_1 \|_2 + 2R_2 \rho(A_2\hat{x}_{k}-b_2,-K)   \leq \frac{2(R_1^2+R_2^2)}{C_k}.
\label{eq:vpmfxh}
\end{equation}
Since  $\lambda^*=(\lambda^{*(1)}, \lambda^{*(2)})^T$ is an optimal solution of Problem $(D_1)$, we have, for any $x \in Q$
$$
Opt[P_1]\leq f(x) + \la  \lambda^{*(1)}, A_1 x -b_1 \ra + \la \lambda^{*(2)}, A_2 x -b_2 \ra.
$$
Using the assumption \eqref{eq:l_bound} and that $\lambda^{*(2)} \in K^*$, we get
\begin{equation}
f(\hat{x}_k) \geq Opt[P_1]- R_1 \|A_1 \hat{x}_k - b_1 \|_2- R_2 \rho(A_2\hat{x}_{k}-b_2,-K) .
\label{eq:fxhat_est}
\end{equation}
Hence,
\begin{align}
 \vp(\eta_k) + f(\hat{x}_k)  & = \vp(\eta_k) - Opt[P_2]+Opt[P_2] + Opt[P_1]  - Opt[P_1] + f(\hat{x}_k)  \notag \\
& \stackrel{\eqref{eq:D_1_P_2_sol}}{=}\vp(\eta_k)  - Opt[P_2]-Opt[D_1]+Opt[P_1]  - Opt[P_1] + f(\hat{x}_k)  
 \notag \\
&  \stackrel{\eqref{eq:wD}}{\geq}  - Opt[P_1] + f(\hat{x}_k) \stackrel{\eqref{eq:fxhat_est}}{\geq} - R_1 \|A_1 \hat{x}_k - b_1 \|_2- R_2 \rho(A_2\hat{x}_{k}-b_2,-K).
\label{eq:aux1}
\end{align}
This and \eqref{eq:vpmfxh} give
\begin{equation}
R_1 \|A_1 \hat{x}_k - b_1 \|_2 + R_2 \rho(A_2\hat{x}_{k}-b_2,-K)  \leq \frac{2(R_1^2+R_2^2)}{C_k}.
\label{eq:R_norm_est}
\end{equation}
Hence, we obtain
\begin{equation}
\vp(\eta_k) + f(\hat{x}_k) \stackrel{\eqref{eq:aux1},\eqref{eq:R_norm_est}}{\geq} - \frac{2(R_1^2+R_2^2)}{C_k}.
\label{eq:vppfxhatgeq}
\end{equation}
On the other hand, we have 
\begin{equation}
\vp(\eta_k) + f(\hat{x}_k) \stackrel{\eqref{eq:vpmfxh}}{\leq} \frac{2(R_1^2+R_2^2)}{C_k}.
\label{eq:vppfxhatleq}
\end{equation}
Combining \eqref{eq:R_norm_est}, \eqref{eq:vppfxhatgeq}, \eqref{eq:vppfxhatleq}, we conclude
\begin{align}
&\|A_1 \hat{x}_k - b_1 \|_2 \leq \frac{2(R_1^2+R_2^2)}{C_kR_1}, \notag \\
&\rho(A_2\hat{x}_{k}-b_2,-K)  \leq \frac{2(R_1^2+R_2^2)}{C_kR_2 }, \notag \\
&|\vp(\eta_k) + f(\hat{x}_k)| \leq \frac{2(R_1^2+R_2^2)}{C_k}.
\label{eq:untileq}
\end{align}
From Lemma \ref{Lm:CkGrowth}, for any $k\geq 0$, $C_k=\frac{(k+1)^2}{8L}$.
Hence, in accordance to \eqref{eq:untileq}, when the iteration counter $k$ is equal to the number given in the theorem statement the stopping criterion fulfills and Algorithm \ref{Alg:PDASTM} stops.

Now let us prove the second statement of the theorem. 
We have 
\begin{align}
 \vp(\eta_k) + Opt[P_1] & = \vp(\eta_k)  - Opt[P_2]+Opt[P_2]+Opt[P_1] \notag \\
& \stackrel{\eqref{eq:D_1_P_2_sol}}{=}  \vp(\eta_k)  - Opt[P_2]-Opt[D_1]+Opt[P_1] \stackrel{\eqref{eq:wD}}{\geq} 0. 
\notag
\end{align} 
Hence,
\begin{equation}
f(\hat{x}_k)- Opt[P_1] \leq f(\hat{x}_k) + \vp(\eta_k).
\label{eq:fxhatmopt}
\end{equation}
On the other hand, 
\begin{equation}
f(\hat{x}_k) - Opt[P_1] \stackrel{\eqref{eq:fxhat_est}}{\geq}  - R_1 \|A_1 \hat{x}_k - b_1 \|_2- R_2 \rho(A_2\hat{x}_{k}-b_2,-K) .
\label{eq:fxhat_est1}
\end{equation}
Note that, since the point $\hat{x}_k$ may not satisfy the linear constraints, one can not guarantee that $f(\hat{x}_k) - Opt[P_1] \geq 0$.
From \eqref{eq:fxhatmopt}, \eqref{eq:fxhat_est1} we can see that if we set $\tilde{\e}_f = \e_f$, $\tilde{\e}_{eq} = \min\{\frac{\e_f}{2R_1},\e_{eq}\}$, $\tilde{\e}_{in} = \min\{\frac{\e_f}{2R_2},\e_{in}\}$ and run Algorithm \ref{Alg:PDASTM} for the number of iterations given in the theorem statement, we obtain that \eqref{eq:sol_def} fulfills and $\hat{x}_k$ is an approximate solution to Problem $(P_1)$ in the sense of \eqref{eq:sol_def}. 
\end{proof}

\section{Numerical Experiments}
\label{S:Num}
In this section, we focus on the problem \eqref{eq:ROT}, which is motivated by important applications to traffic demand matrix estimation, \cite{wilson2011entropy}, and regularized optimal transport calculation, \cite{cuturi2013sinkhorn}. We provide the results of our numerical experiments, which were performed on a PC with processor Intel Core i5-2410 2.3 GHz and 4 GB of RAM using pure Python 2.7 (without C code) under managing OS Ubuntu 14.04 (64-bits). Numpy.float128 data type with precision $1e-18$  and with max element $\approx 1.19e+4932$ was used. No parallel computations were used.

\subsection{Setting}
We compare the performance of our algorithm with Sinkhorn's-method-based approach of \cite{cuturi2013sinkhorn}, which is the state-of-the art method for problem \eqref{eq:ROT}. We use two types of cost matrix $C$ and three types of vectors $\mu$ and $\nu$. 

\textbf{Cost matrix $C$.} The first type of the cost matrix $C$ is usually used in optimal transport problems and corresponds to 2-Wasserstein distance. Assume that we need to calculate this distance between two discrete measures $\mu,\nu$ with finite support of size $p$. Then, the element $c_{ij}$ of the matrix $C$ is equal to Euclidean distance between the $i$-th point in the support of the measure $\mu$ and $j$-th point in the support of the measure $\nu$. We will refer to this choice of the cost matrix as \textit{Euclidean cost}.
The second type the cost matrix $C$ comes from traffic matrix estimation problem. Let's consider a road network of Manhattan type, i.e. districts present a $m \times m$ grid. We build a $m^2$ by $m^2$ matrix $D$ of pairwise Euclidian distances processing the grid rows one by one and calculating euclidean distances from the current grid element to all the others elements of the grid.
Then, as it suggested in \cite{shvetsov2003mathematical}, we form the cost matrix $C$ as $C = \exp (-0.065 D)$, where the exponent is taken elementwise. We will refer to this choice of the cost matrix as \textit{Exp-Euclidean cost}.

To set a natural scale for the regularization parameter $\gamma$, we normalize in each case the matrix $C$ dividing all its elements by the average of all elements.

\textbf{Vectors $\mu$ and $\nu$.}
The first type of vectors $\mu$ and $\nu$ is \textit{normalized uniform random}. Each element of each vector is taken independently from the uniform distribution on $[0,1]$ and then each vector is normalized so that each sums to 1, i.e. 
%

The second type of vectors is \textit{random images}. The first $p/2$ elements of $\mu$ are normalized uniform random and the second $p/2$ elements are zero. For $\nu$ the situation is the opposite, i.e. the first $p/2$ elements are zero, and the second $p/2$ elements are normalized uniform random. In our preliminary experiments we found that the methods behave strange on vectors representing pictures from MNIST database (see below). We supposed that the reason is that these vectors have many zero elements and decided to include the described random images to the experiments setting.
%
%
Finally, the third type are vectors of intensities of \textit{images} of handwritten digits from MNIST dataset. The size of each image is 28 by 28 pixels. Each image is converted to gray scale from 0 to 1 where 0 corresponds to black color and 1 corresponds to white, then each image is reshaped to a vector of length 784. In our experiments, we normalize these vectors to sum to 1. The elements of MNIST look like:

\begin{figure}[H]
	\centering
        \begin{subfigure}[b]{0.08\textwidth}
                \includegraphics[width=\linewidth]{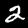}
        \end{subfigure}%
        \begin{subfigure}[b]{0.08\textwidth}
                \includegraphics[width=\linewidth]{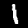}
        \end{subfigure}%
        \begin{subfigure}[b]{0.08\textwidth}
                \includegraphics[width=\linewidth]{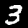}
        \end{subfigure}%
        \begin{subfigure}[b]{0.08\textwidth}
                \includegraphics[width=\linewidth]{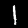}
        \end{subfigure}%
        \begin{subfigure}[b]{0.08\textwidth}
                \includegraphics[width=\linewidth]{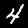}
        \end{subfigure}%
        \caption{Examples of images in MNIST dataset.}
				\label{fig:mnist}
\end{figure}

\textbf{Accuracy.}
We slightly redefine the accuracy of the solution and use relative accuracy with respect to the starting point, i.e.
$$
\tilde{\e}_f = [\text{Accuracy}] \cdot |f(x(\lambda_0))+\vp(\eta_0)|, \quad  \tilde{\e}_{eq} = [\text{Accuracy}] \cdot \|A_1x(\lambda_0)-b_1\|_2,
$$
where we used the fact that $\lambda_0=\eta_0=0$ and there are no cone constraints in \eqref{eq:ROT}.

\subsection{Preliminary Experiments}
\textbf{Adaptive vs non-adaptive algorithm.} First, we show that the adaptivity of our algorithm with respect to the Lipschitz constant of the gradient of $\vp$ leads to faster convergence in practice. For this purpose, we use normalized uniform random vectors $\mu$ and $\nu$ and both types of cost matrix $C$. We compare our new Algorithm \ref{Alg:PDASTM} with non-adaptive Similar Triangles Method (STM), which has cheaper iteration than the existing non-adaptive methods \cite{patrascu2015rate,gasnikov2016efficient,chernov2016fast,li2016inexact}.
We choose $m = 10$, and, hence, $p$ = 100, Accuracy is 0.05. For the Exp-Euclidean cost matrix $C$, we use $\gamma \in \{0.1, 0.2, 0.3, 0.4, 0.5\}$, and, for Euclidean cost matrix $C$, we use $\gamma \in \{0.02, 0.1, 0.2, 0.3, 0.4, 0.5\}$. The results are shown in Figure \ref{fig:PDASTMvsSTM}. In both cases our new Algorithm \ref{Alg:PDASTM} is much faster than the STM. This effect was observed for other parameter values, so, in the following experiments, we consider PDASTM.

\begin{figure}[H]
	\centering
	\begin{subfigure}[b]{0.4\textwidth}
		\includegraphics[width=\textwidth]{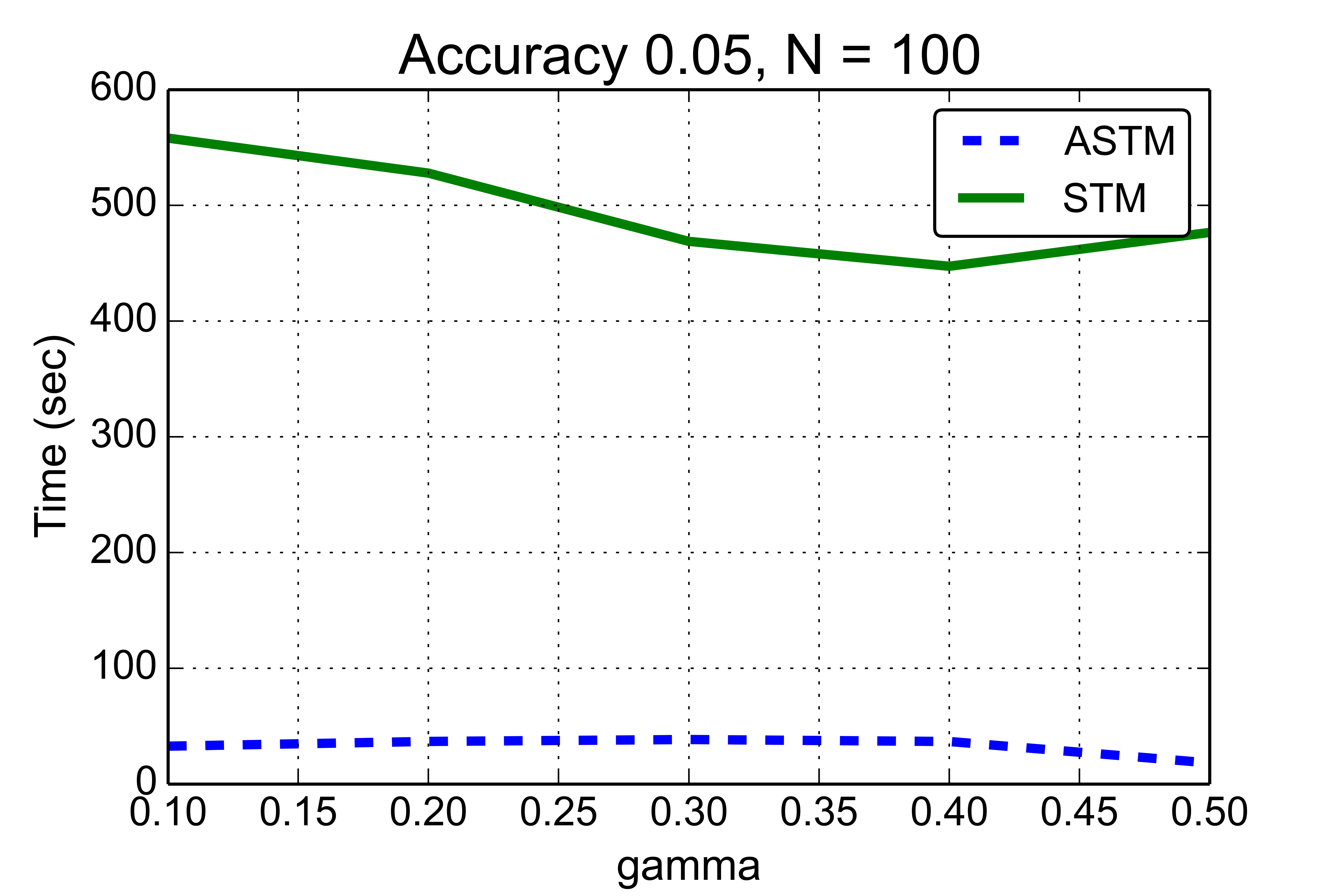}
	\end{subfigure}
	\begin{subfigure}[b]{0.4\textwidth}
		\includegraphics[width=\textwidth]{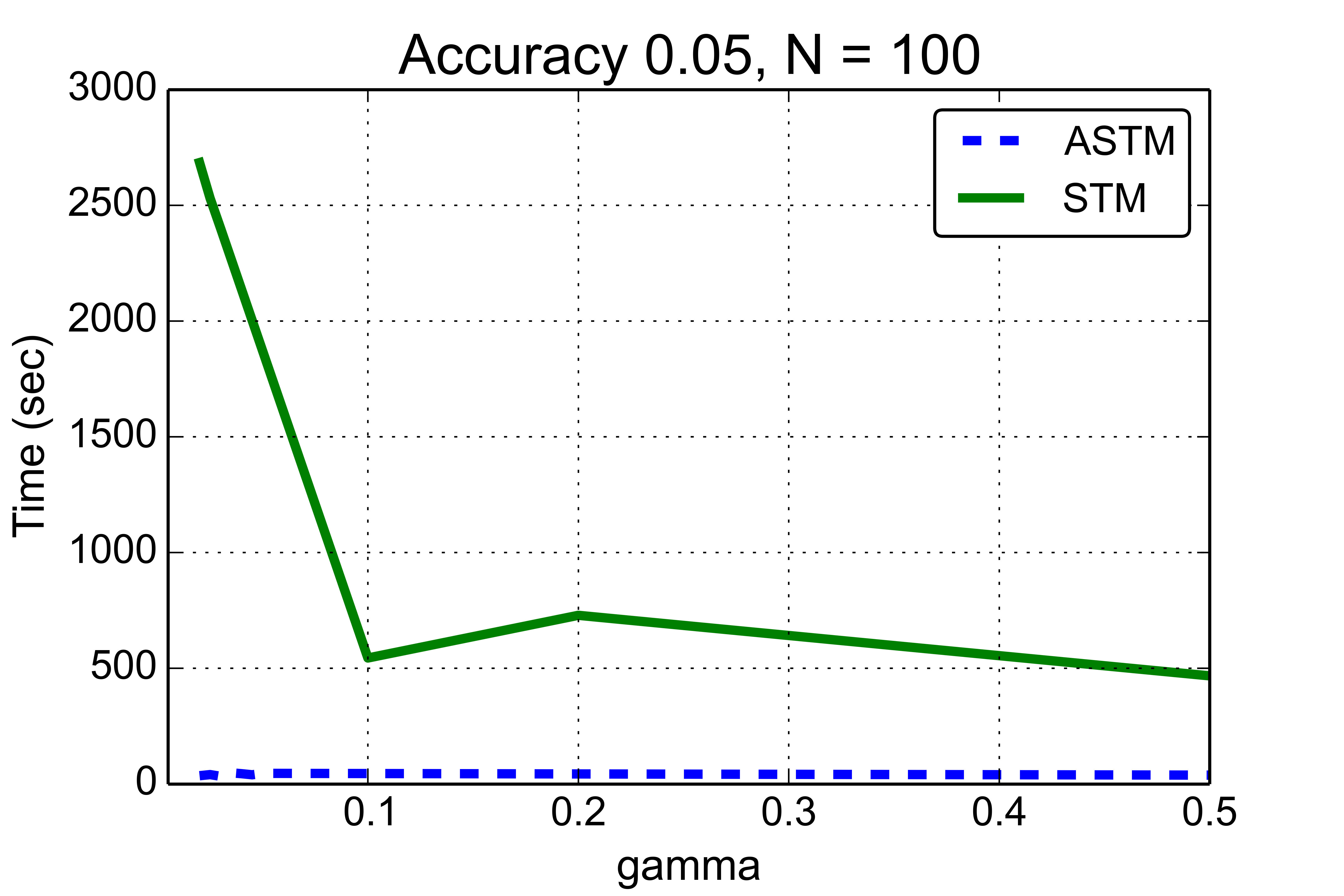}
	\end{subfigure}
	\caption{The perfomance of PDASTM vs STM, Accuracy 0.05, Exp-Euclidean $C$ (left) and Euclidean $C$ (right).}
	\label{fig:PDASTMvsSTM}
\end{figure}

\textbf{Warm start.} 
During our experiments on the images from  MNIST dataset PDASTM worked worse than on the normalized uniform random vectors. Possible reason is the large number of zero elements in the former vectors (a lot of black pixels). So we decided to test the performance of the algorithms on the random images vectors $\mu$ and $\nu$. Also we decided to apply the idea of warm start to force PDASTM to converge faster. As we know, Sinkhorn's method works very fast when $\gamma$ is relatively large. Thus, we use it in this regime to find a good starting point for the PDASTM for the problem with small $\gamma$. Notably, the running time of Sinkhorn's method is small in comparison with time of ASTM running.
We test the performance of PDASTM versus PDASTM with warm start on problems with Exp-Euclidean matrix $C$ and $\gamma \in \{0.001, 0.003, 0.005, 0.008, 0.01\}$ and on problems with Euclidean matrix $C$ and $\gamma \in \{0.005, 0.01, 0.015, 0.02, 0.025\}$. The results are in Figure \ref{fig:PDASTMvsPDASTMwWS}. Other parametersare stated in the figure. The experiments were run 7 times, the results were averaged. 
As we can see, warm start accelerates the PDASTM. Similar results were observed in other experiments, so, we made the final comparison between the Sinkhorn's method and PDASTM with warm start. 

\begin{figure}[H]
	\centering
	\begin{subfigure}[b]{0.4\textwidth}
		\includegraphics[width=\textwidth]{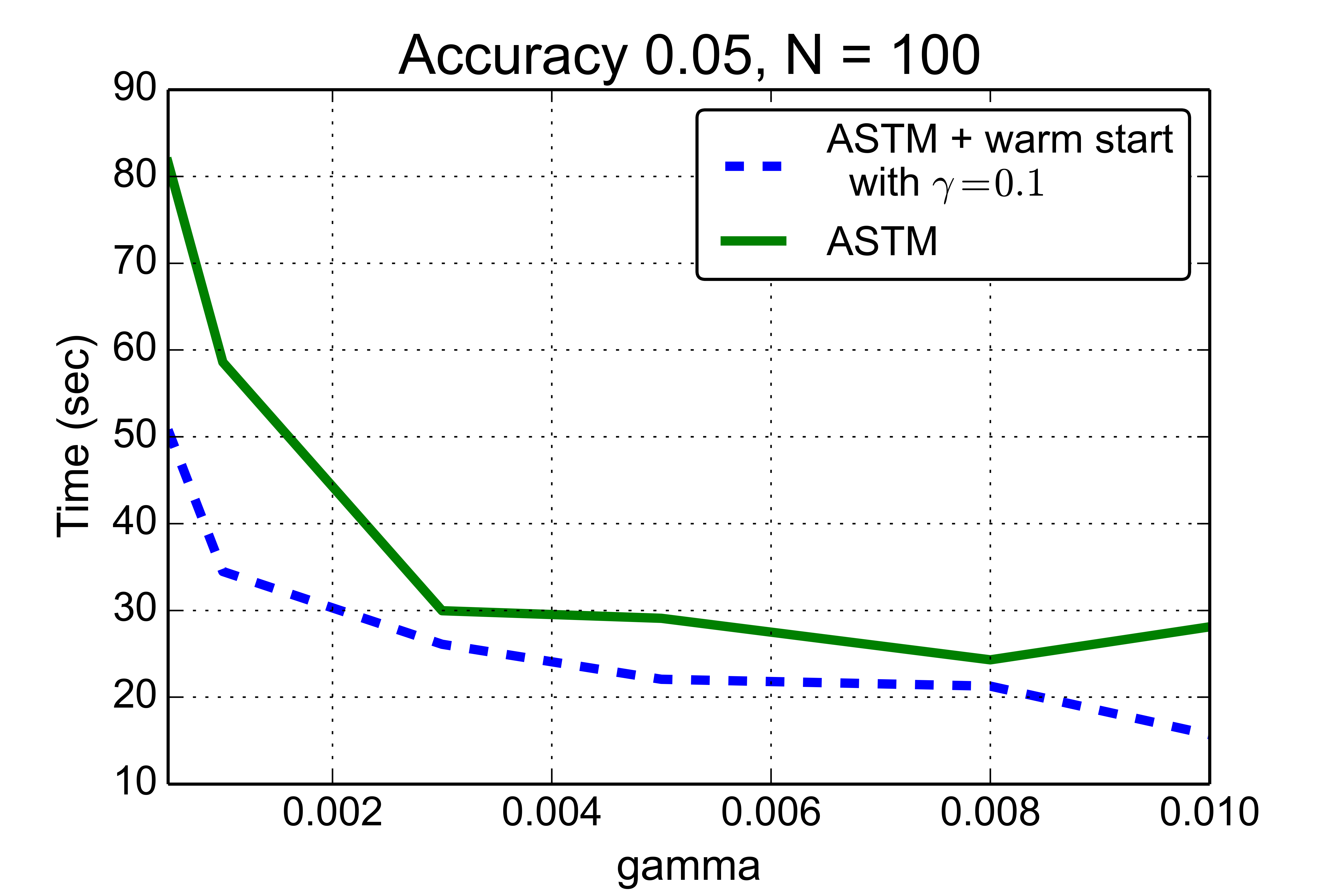}
	\end{subfigure}
	\begin{subfigure}[b]{0.4\textwidth}
		\includegraphics[width=\textwidth]{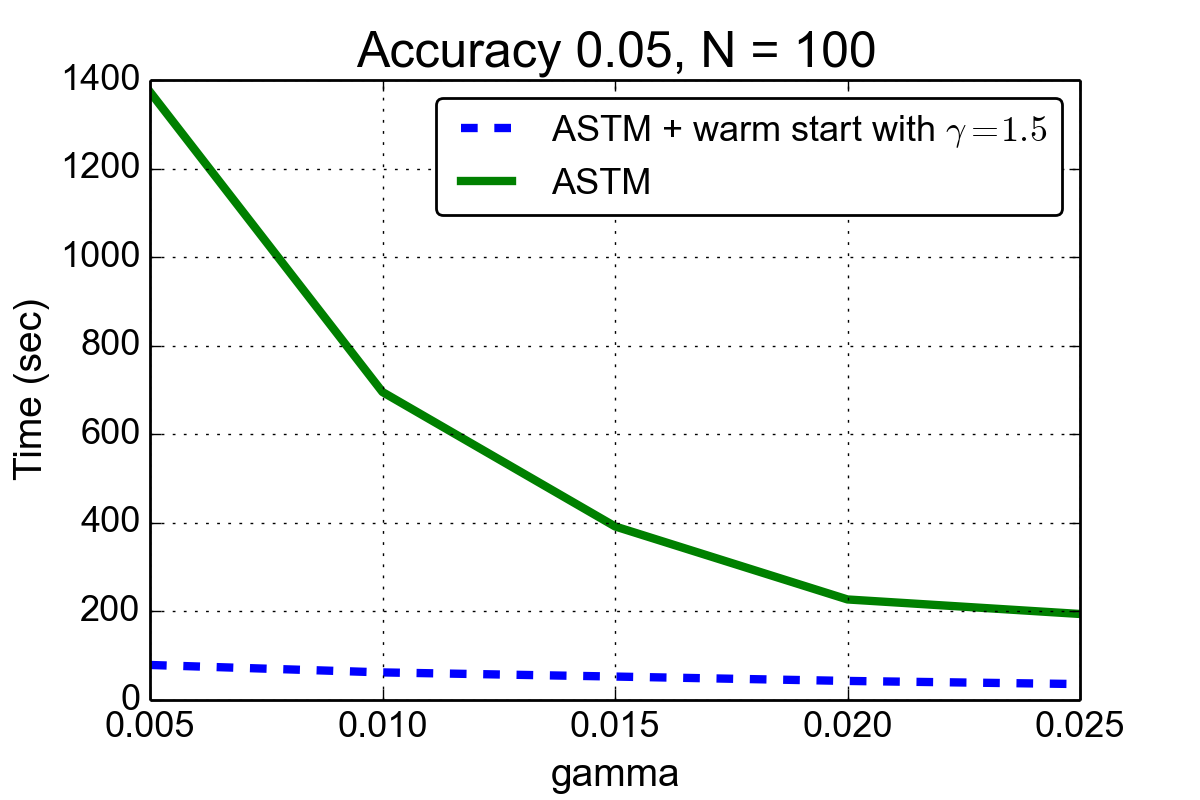}
	\end{subfigure}
	\caption{The perfomance of PDASTM vs PDASTM with warm start, Accuracy 0.05, Exp-Euclidean $C$ (left) and Euclidean $C$ (right).}
		\label{fig:PDASTMvsPDASTMwWS}
\end{figure}

\subsection{Sinkhorn's Method vs PDASTM with Warm Start}
First we compare Sinkhorn's method and PDASTM with warm start on the problem with normalized uniform random vectors $\mu$, $\nu$ and Euclidean cost matrix $C$ with different values of $p \in \{100, 196, 289, 400\}$, $\text{Accuracy} \in \{0.01, 0.05, 0.1\}$, and $\gamma \in [0.005; 0.025]$. On each graph we point the value of $\gamma$ used for generating a starting point for PDASTM with warm start by Sinkhorn's method. Each experiments was run 5 times and then the results were averaged. The results are shown on the Figures \ref{fig:varwn_astm_01}, \ref{fig:varwn_astm_005}, \ref{fig:varwn_astm_001}.

\begin{figure}[H]	
	\centering
	\begin{subfigure}[b]{0.4\textwidth}
		\includegraphics[width=\textwidth]{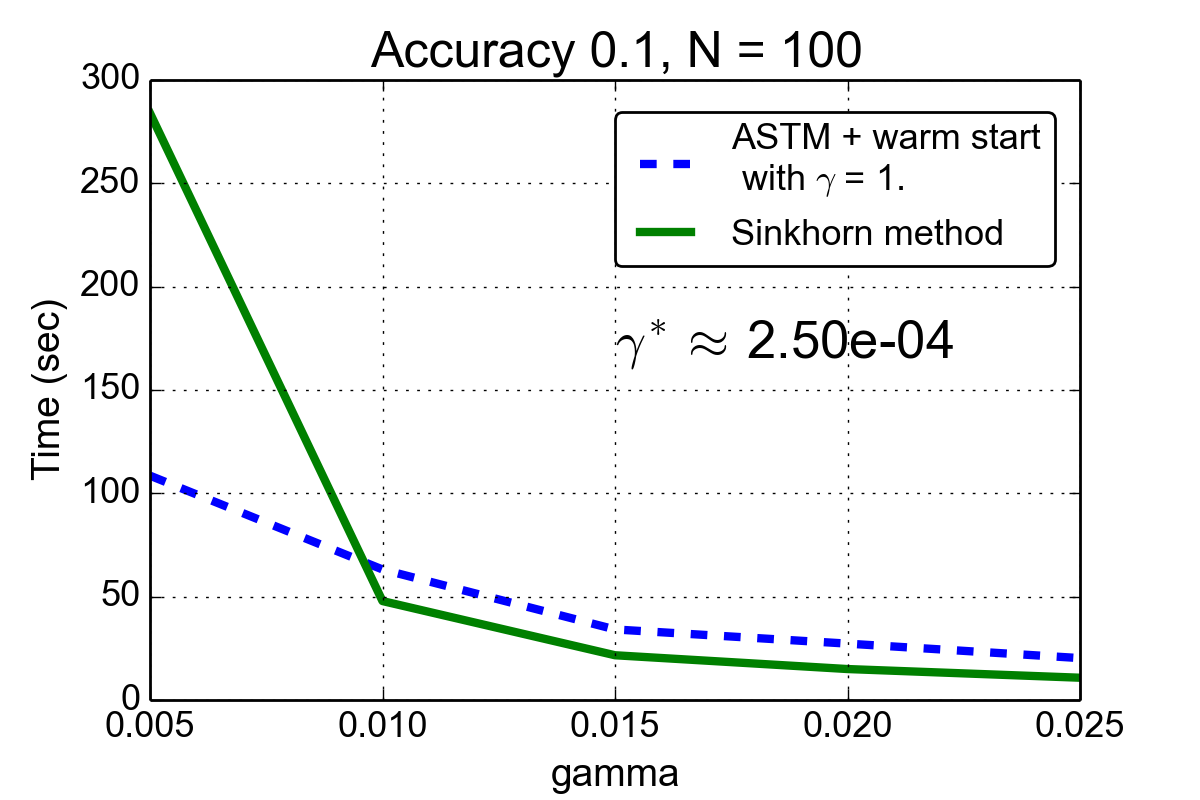}
	\end{subfigure}
	\begin{subfigure}[b]{0.4\textwidth}
		\includegraphics[width=\textwidth]{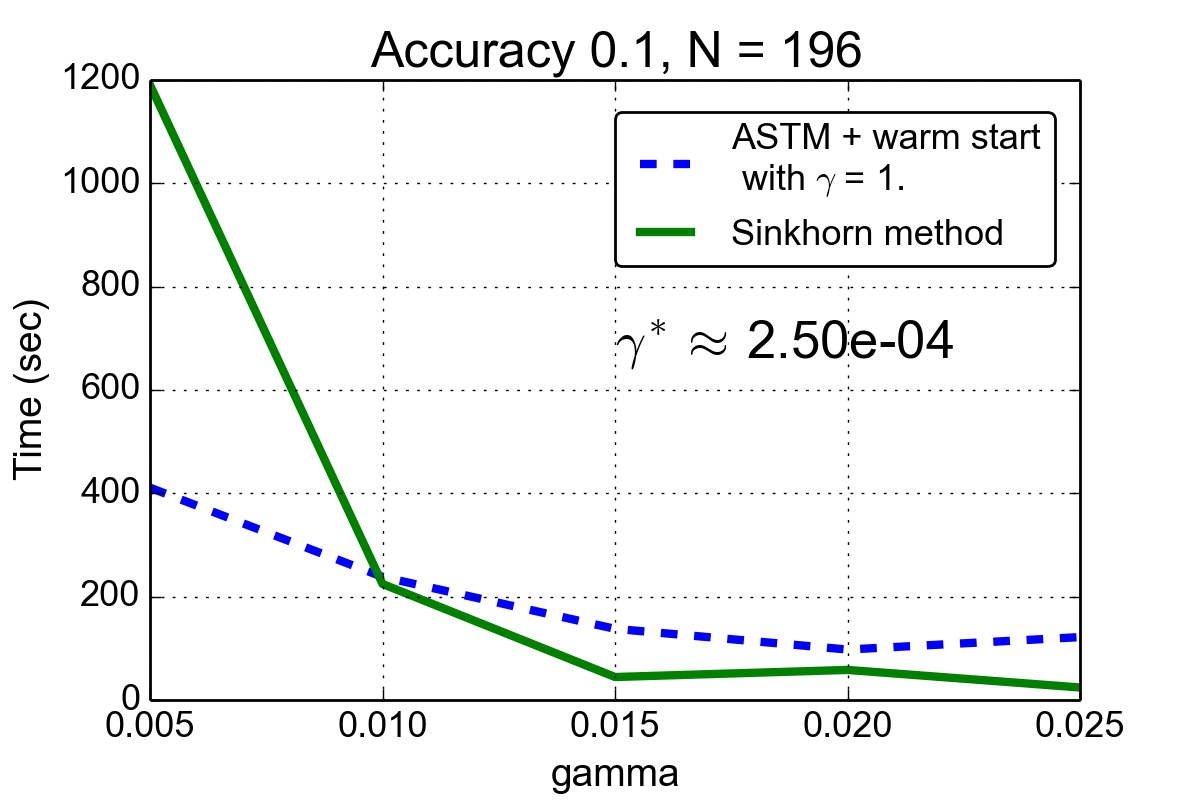}
	\end{subfigure}

	\begin{subfigure}[b]{0.4\textwidth}
		\includegraphics[width=\textwidth]{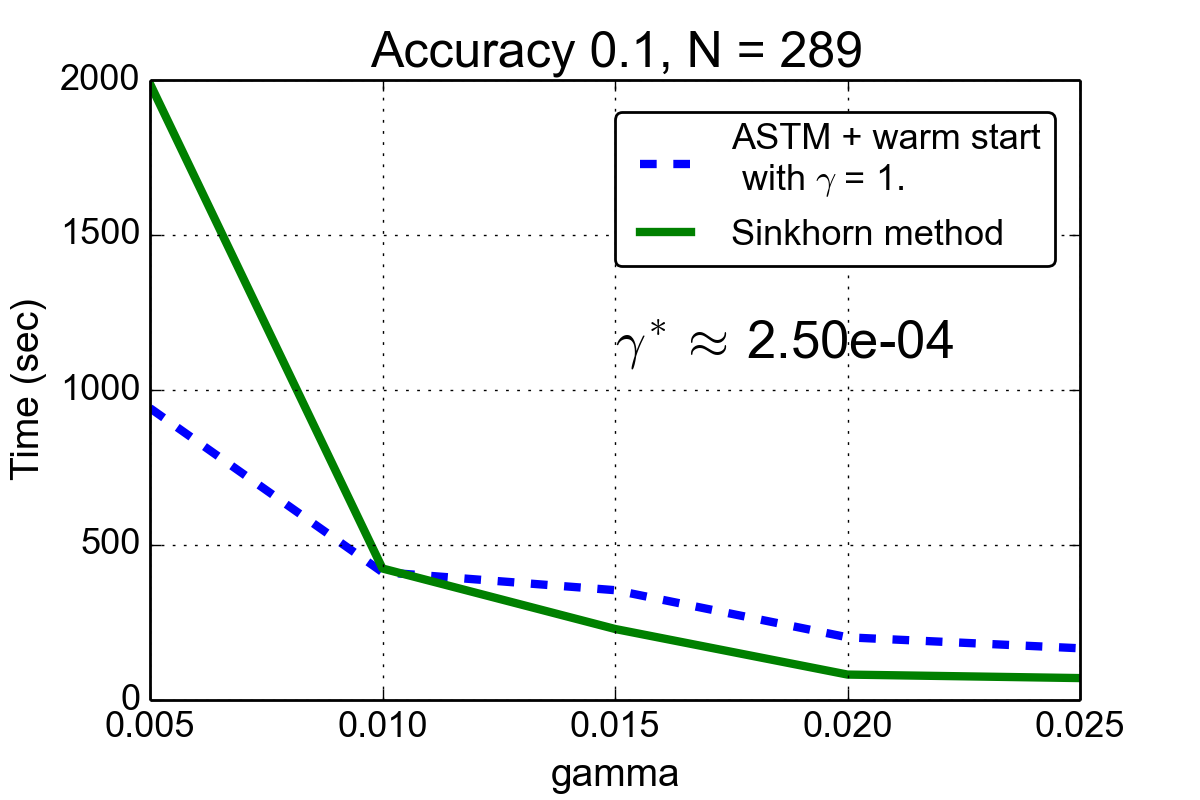}
	\end{subfigure}
	\begin{subfigure}[b]{0.4\textwidth}
		\includegraphics[width=\textwidth]{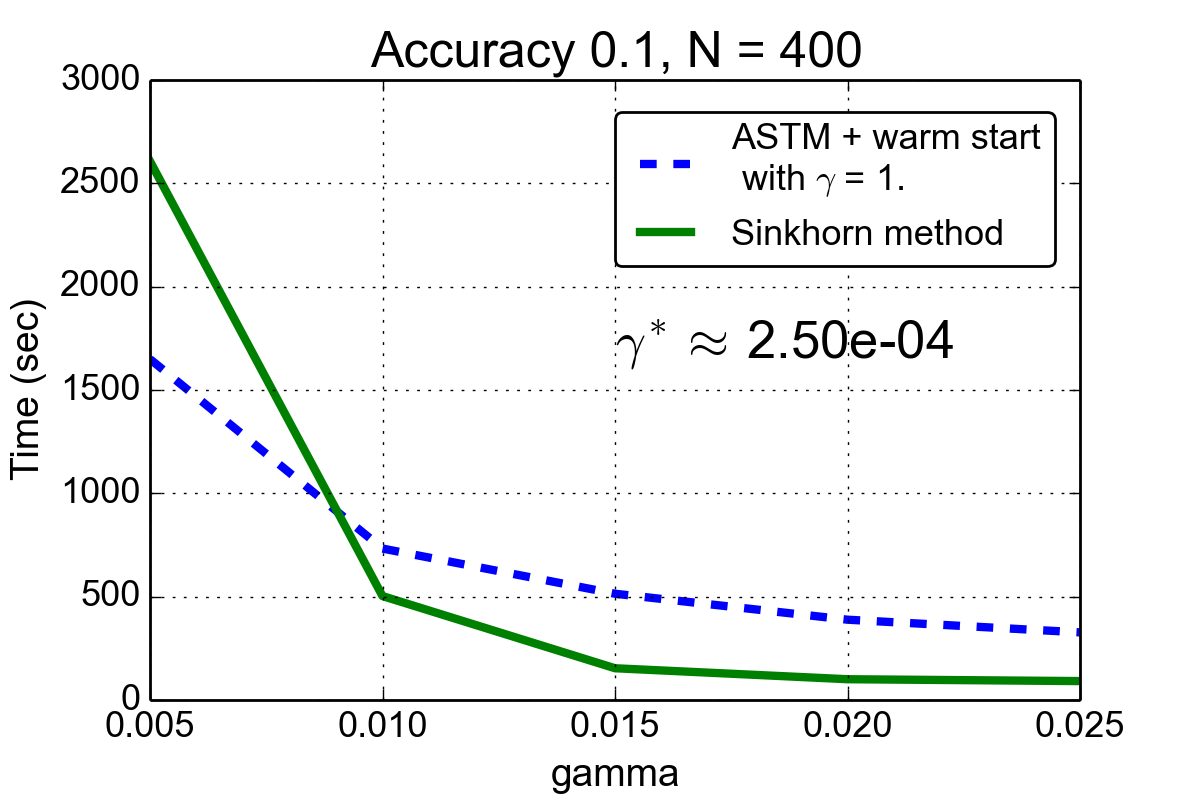}
	\end{subfigure}
	\caption{The perfomance of PDASTM with warm start vs Sinkhorn's method, Accuracy 0.1, Euclidean cost matrix $C$.}
	\label{fig:varwn_astm_01}
\end{figure}

\begin{figure}[H]	
	\centering
	\begin{subfigure}[b]{0.4\textwidth}
		\includegraphics[width=\textwidth]{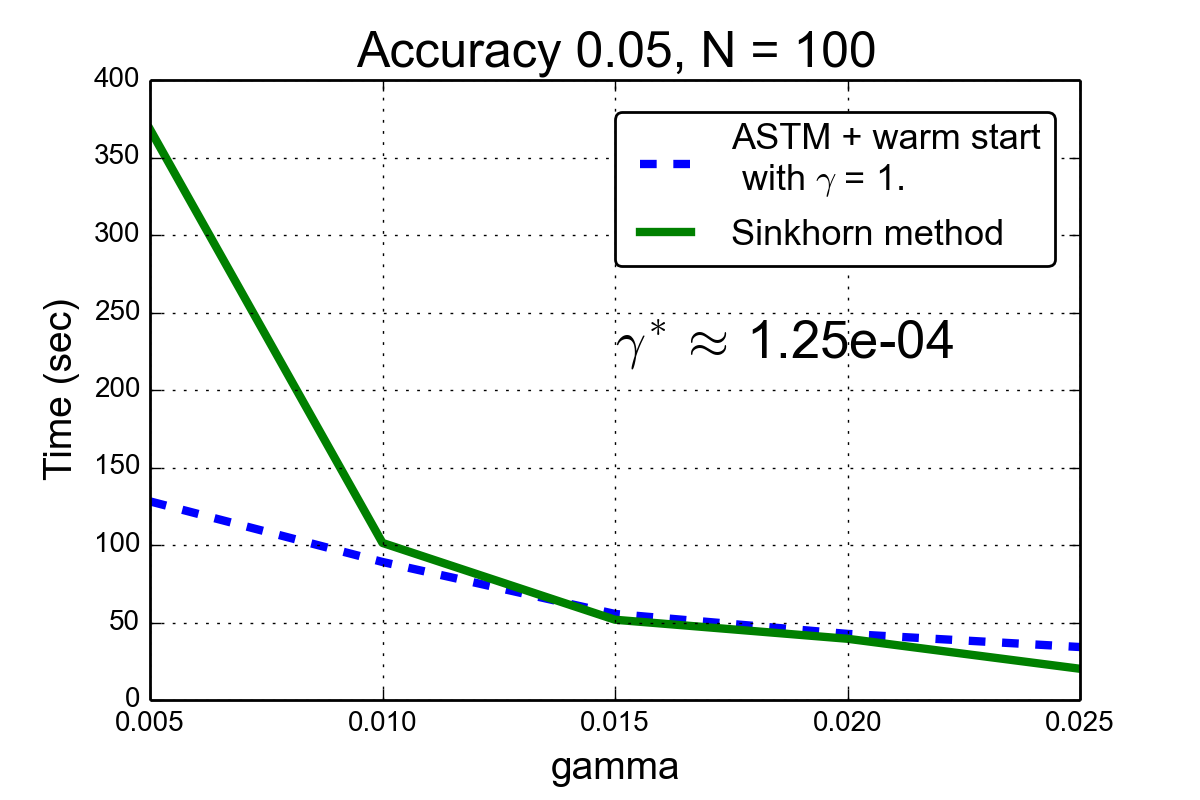}
	\end{subfigure}
	\begin{subfigure}[b]{0.4\textwidth}
		\includegraphics[width=\textwidth]{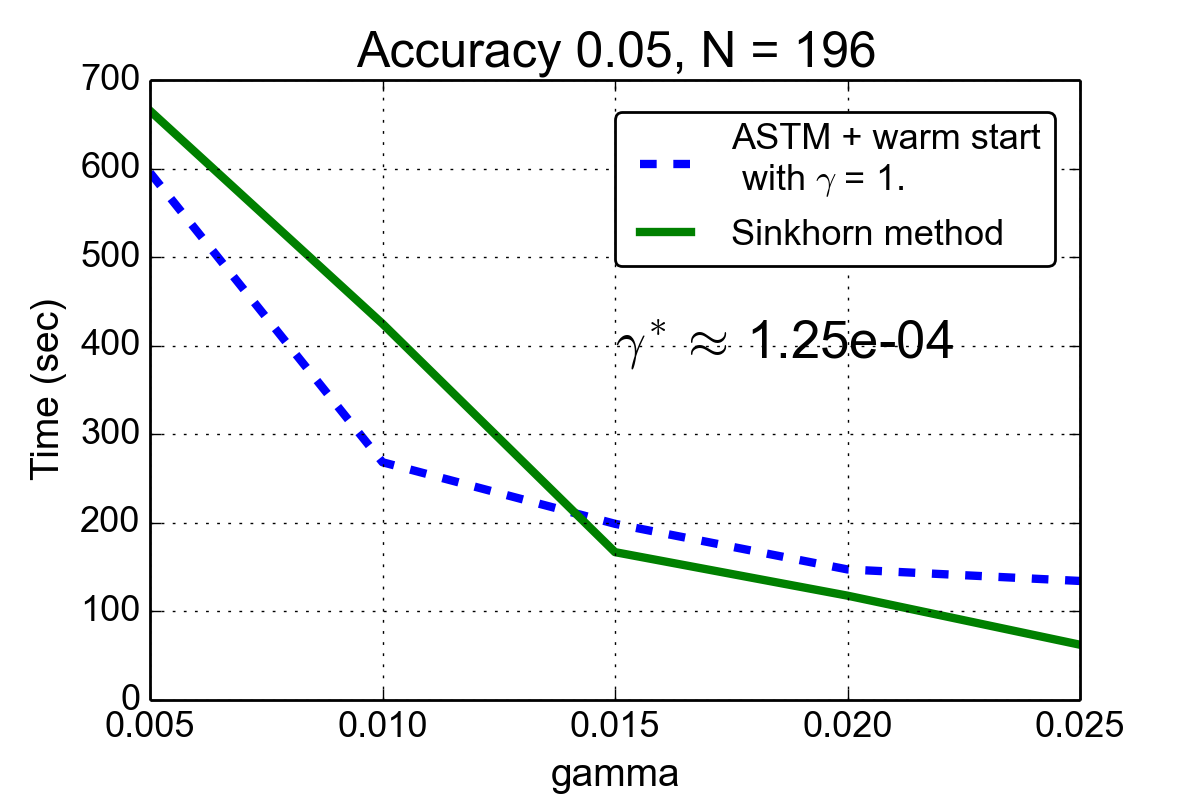}
	\end{subfigure}

	\begin{subfigure}[b]{0.4\textwidth}
		\includegraphics[width=\textwidth]{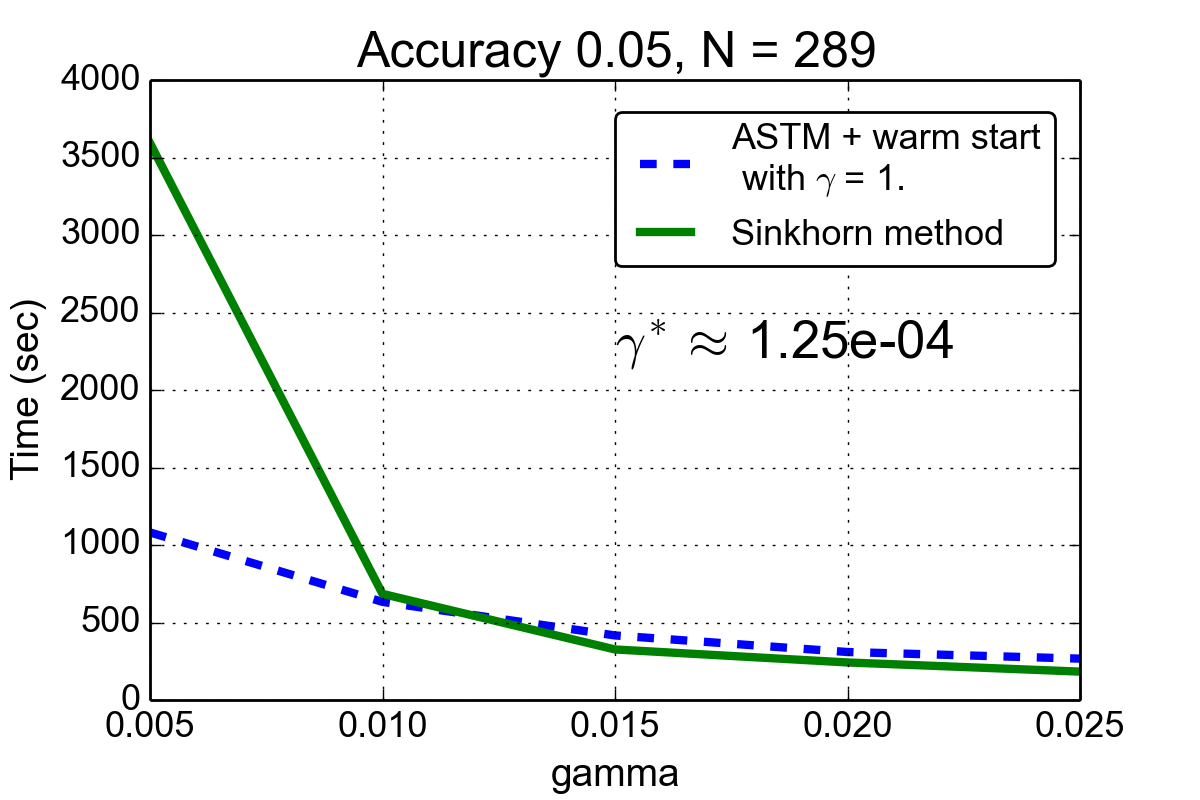}
	\end{subfigure}
	\begin{subfigure}[b]{0.4\textwidth}
		\includegraphics[width=\textwidth]{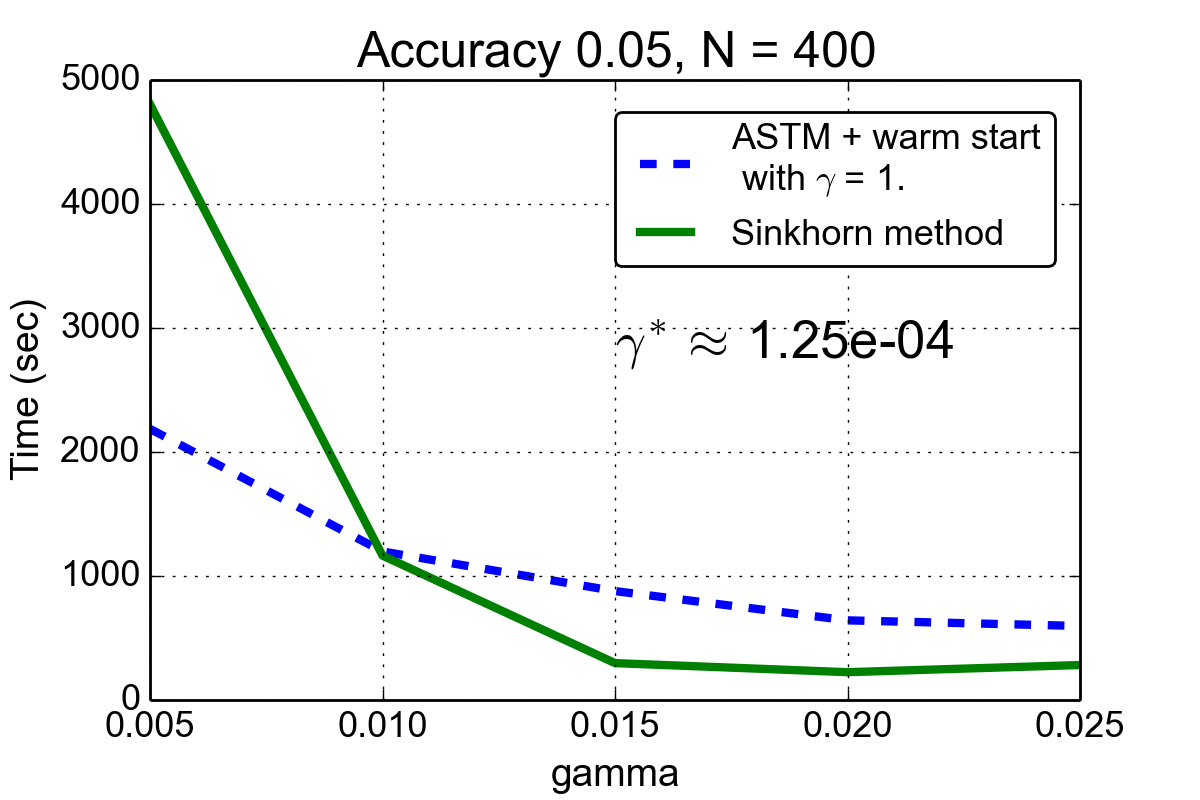}
	\end{subfigure}
	\caption{The perfomance of PDASTM with warm start vs Sinkhorn's method, Accuracy 0.05, Euclidean cost matrix $C$.}
	\label{fig:varwn_astm_005}
\end{figure}

\begin{figure}[H]
	\centering
	\begin{subfigure}[b]{0.4\textwidth}
		\includegraphics[width=\textwidth]{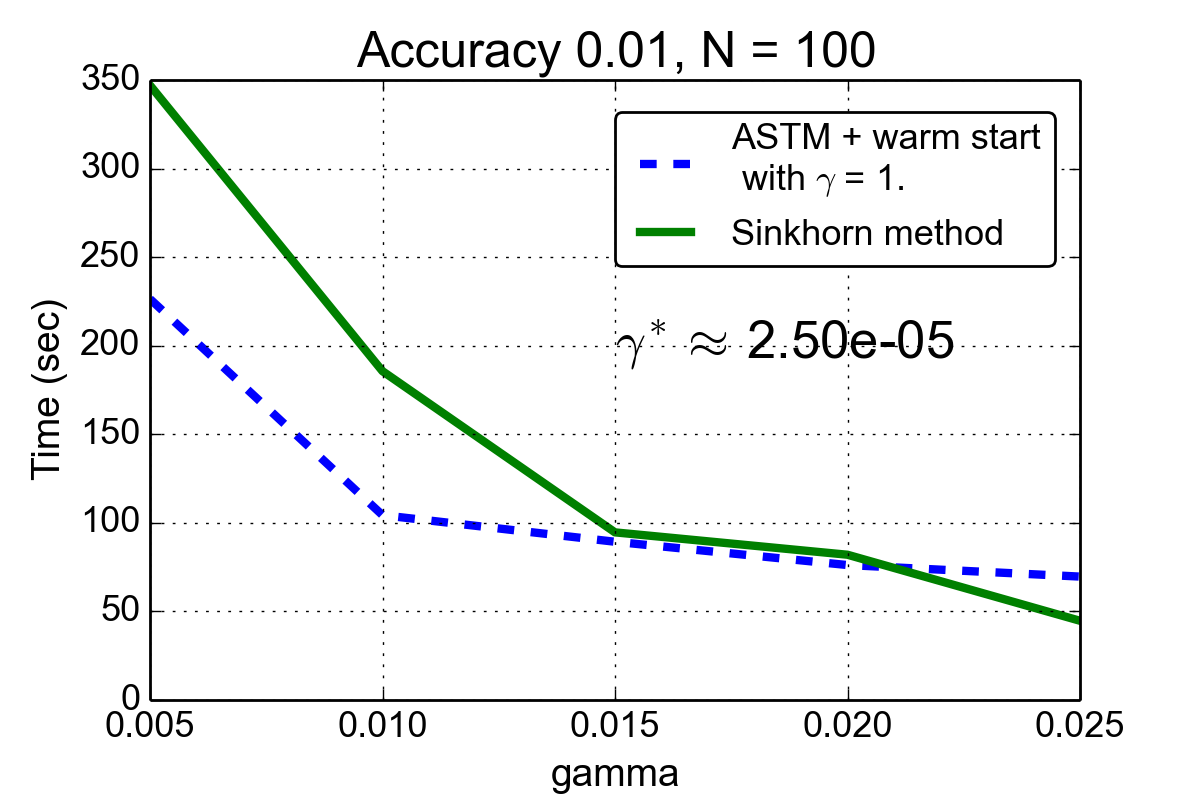}
	\end{subfigure}
	\begin{subfigure}[b]{0.4\textwidth}
		\includegraphics[width=\textwidth]{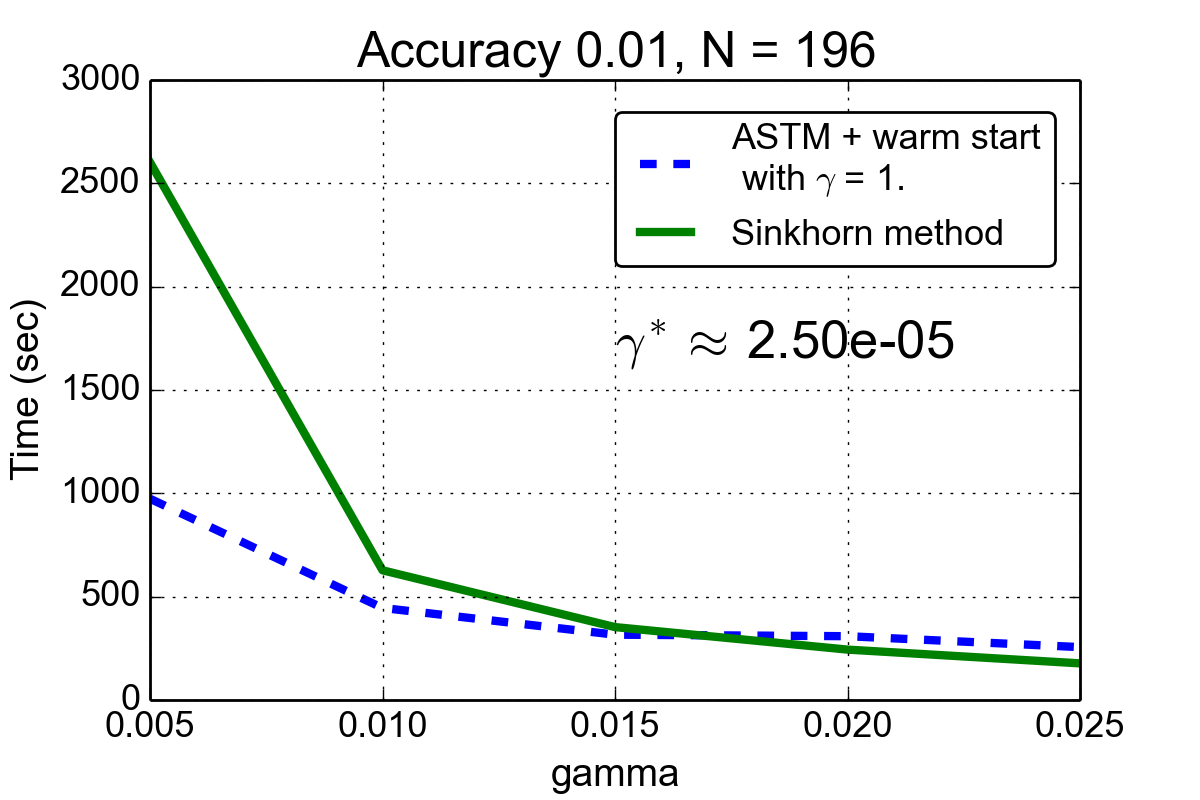}
	\end{subfigure}

	\begin{subfigure}[b]{0.4\textwidth}
		\includegraphics[width=\textwidth]{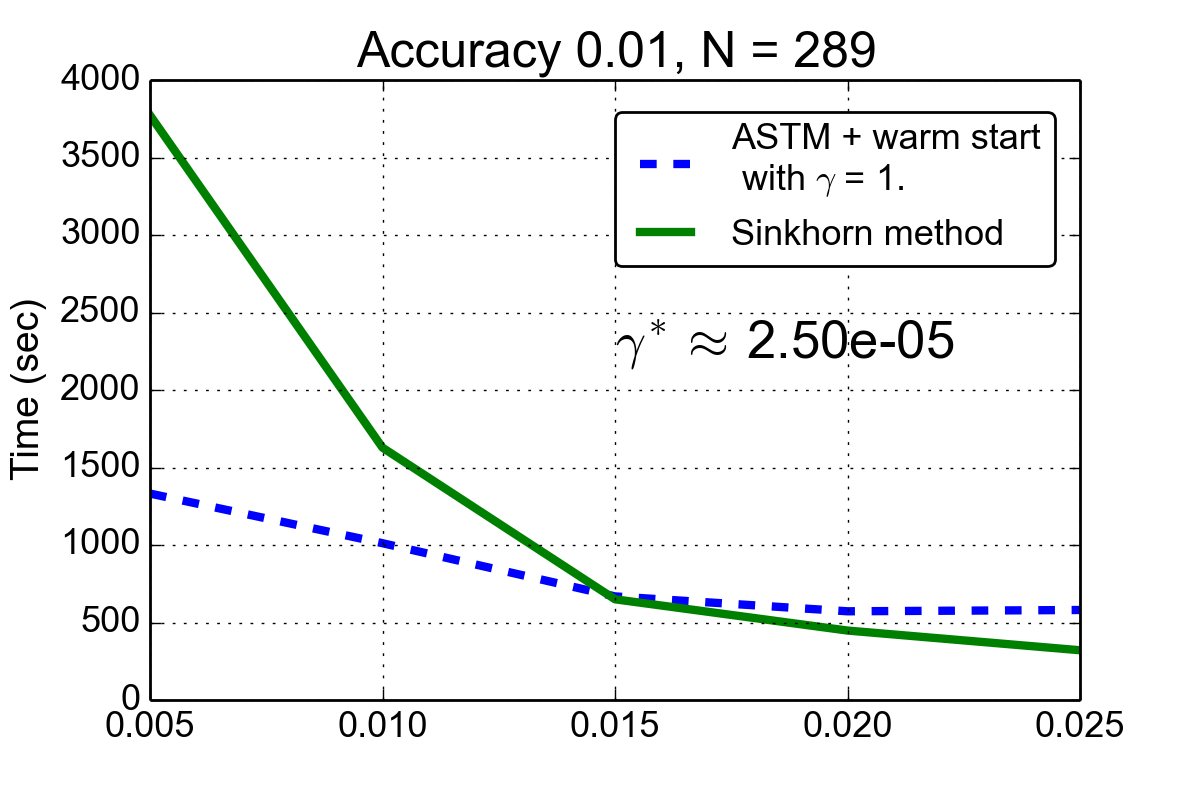}
	\end{subfigure}
	\begin{subfigure}[b]{0.4\textwidth}
		\includegraphics[width=\textwidth]{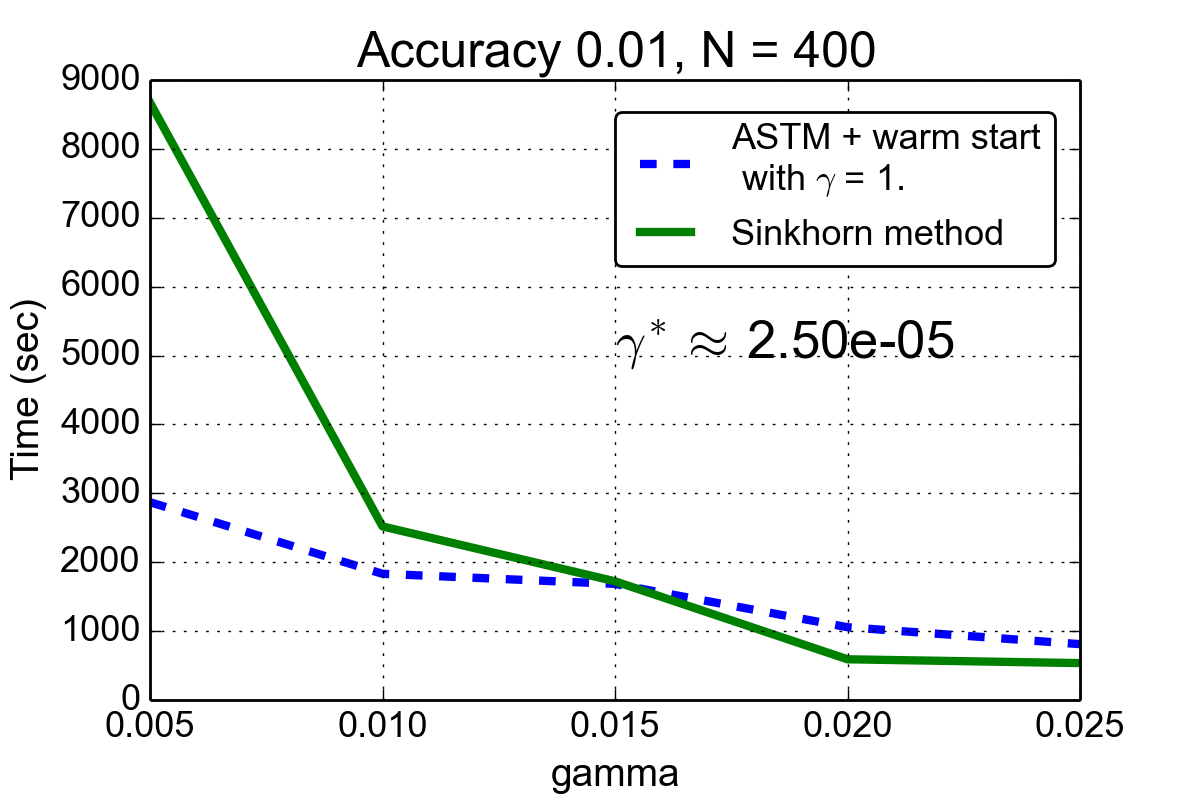}
	\end{subfigure}
	\caption{The perfomance of PDASTM with warm start vs Sinkhorn's method, Accuracy 0.01, Euclidean cost matrix $C$.}
	\label{fig:varwn_astm_001}
\end{figure}

For the Exp-Euclidean cost matrix $C$, we performed the same experiments. For the space reasons, we provide the results on the Figure \ref{fig:varwn_astm_exp_005} only for Accuracy 0.05. The results for other Accuracy values were similar.

\begin{figure}[H]
	\centering
	\begin{subfigure}[b]{0.4\textwidth}
		\includegraphics[width=\textwidth]{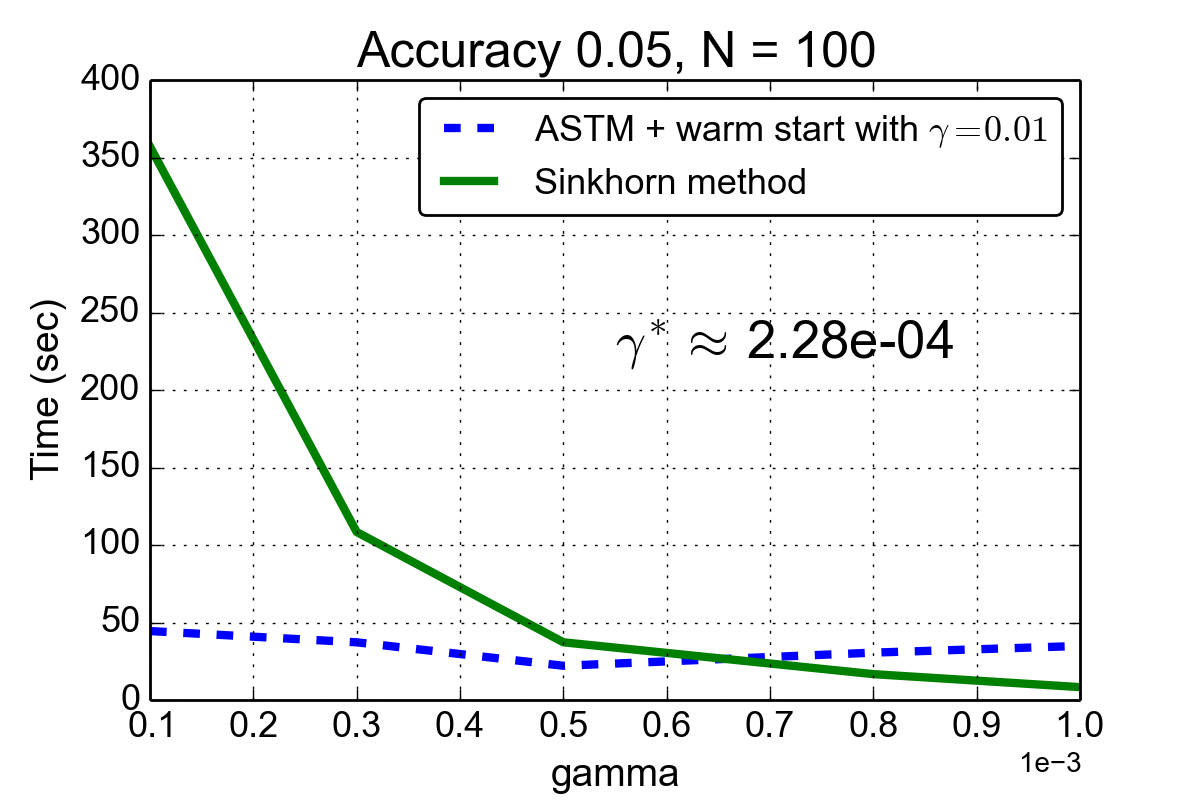}
	\end{subfigure}
	\begin{subfigure}[b]{0.4\textwidth}
		\includegraphics[width=\textwidth]{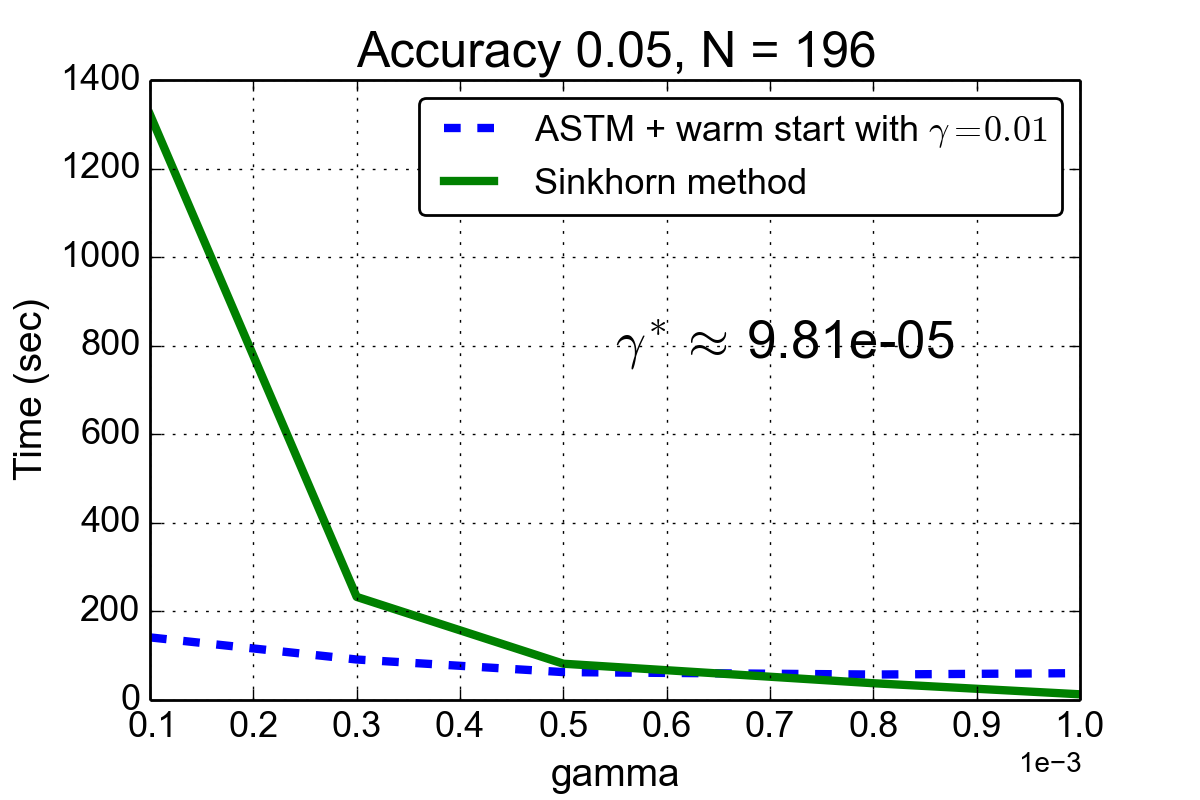}
	\end{subfigure}

	\begin{subfigure}[b]{0.4\textwidth}
		\includegraphics[width=\textwidth]{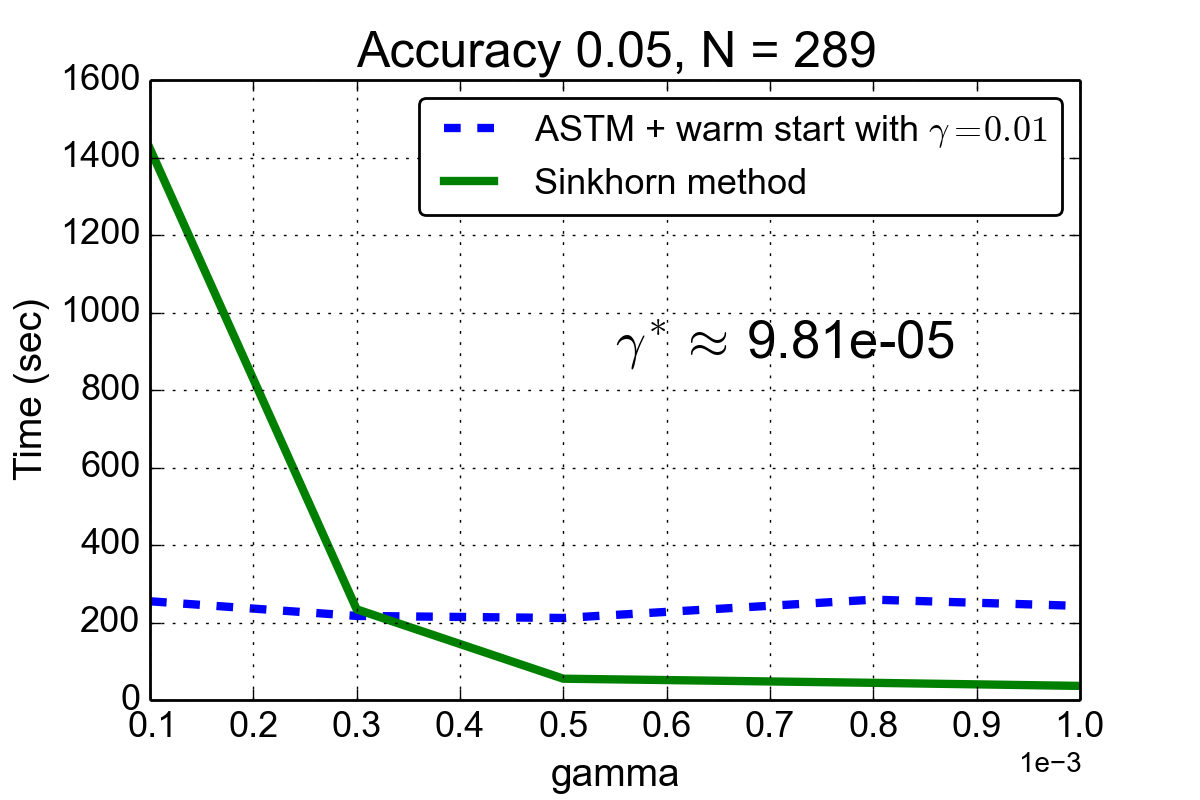}
	\end{subfigure}
	\begin{subfigure}[b]{0.4\textwidth}
		\includegraphics[width=\textwidth]{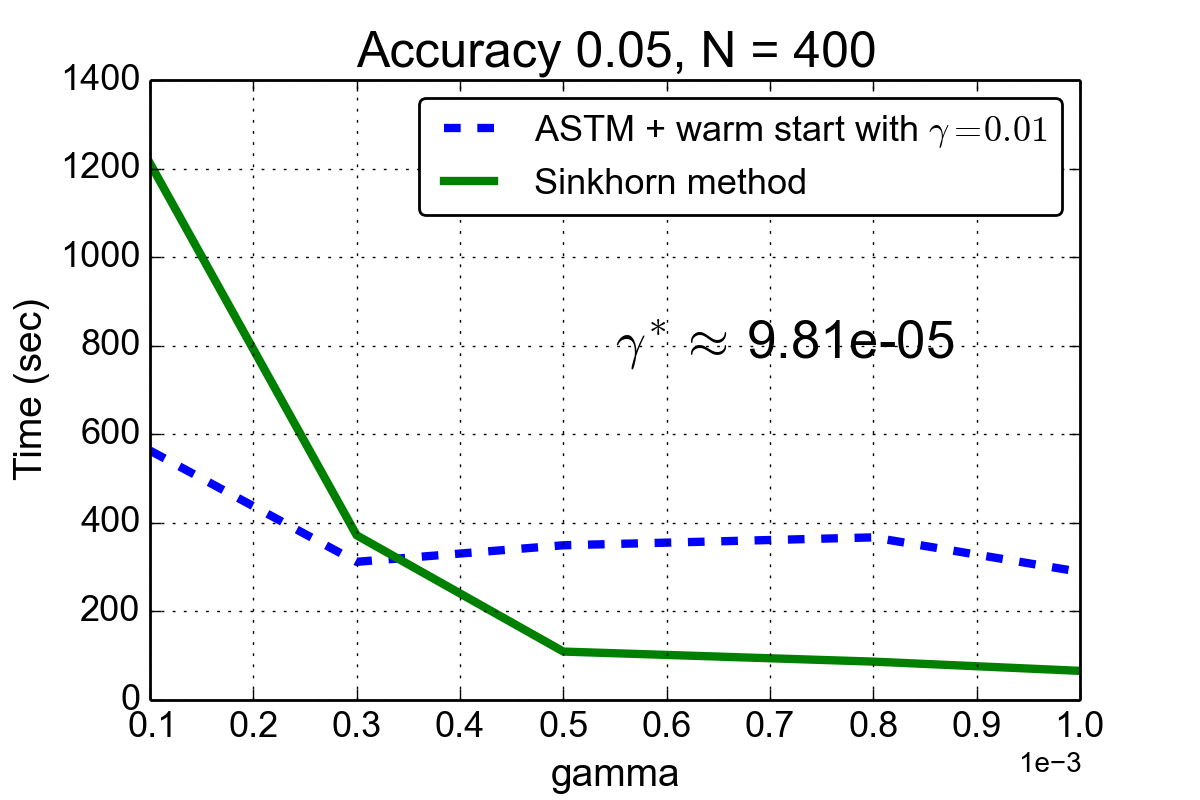}
	\end{subfigure}
	\caption{The perfomance of PDASTM with warm start vs Sinkhorn's method, Accuracy 0.05, Exp-Euclidean cost matrix $C$.}
	\label{fig:varwn_astm_exp_005}
\end{figure}

In another series of experiments we compare the performance of PDASTM with warm start and Sinkhorn's method on the problem with images from MNIST dataset and Euclidean cost matrix $C$. We run both algorithms for the same set of $\gamma$ values for 5 pairs of images. The results are aggregated by $\gamma$ and the performance is averaged for each $\gamma$. We take three values of Accuracy, $\{0.01, 0.05, 0.1\}$. The results are shown on the Figure \ref{fig:varwn_astm_MNIST}.

\begin{figure}[H]	
	\centering
	\begin{subfigure}[b]{0.4\textwidth}
		\includegraphics[width=\textwidth]{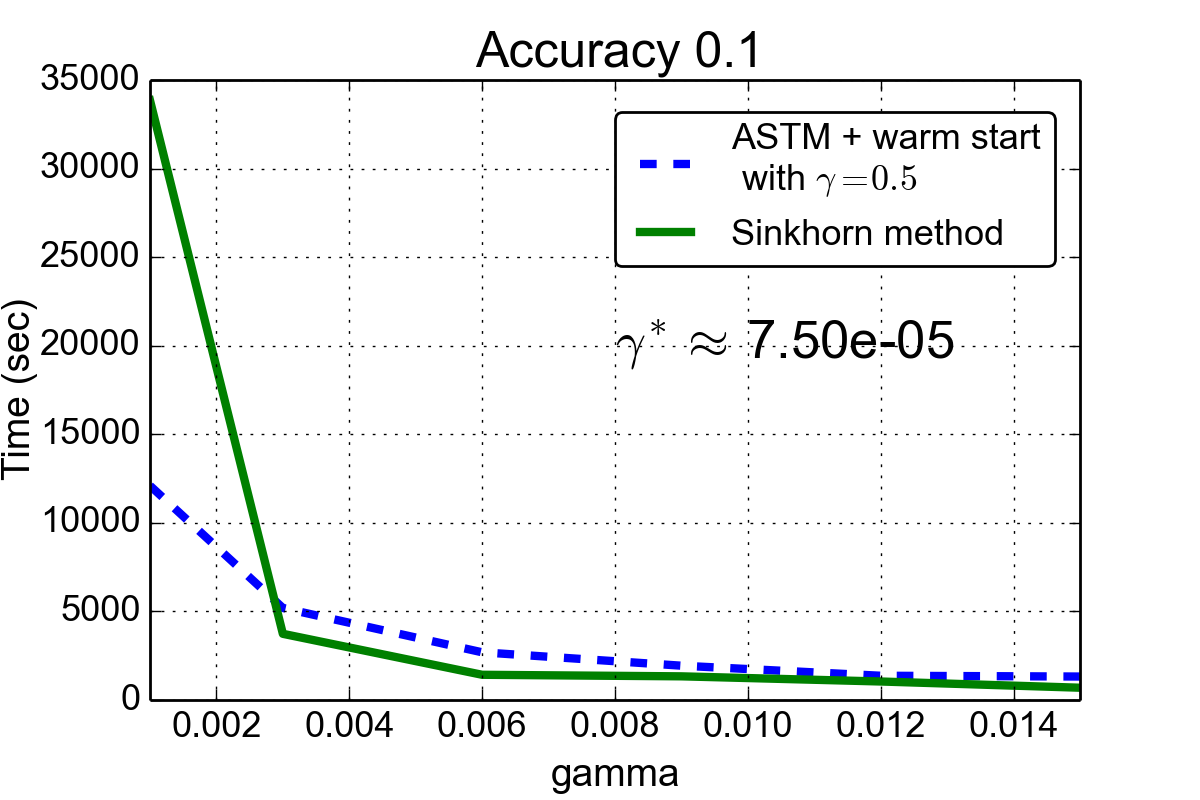}
	\end{subfigure}
	\begin{subfigure}[b]{0.4\textwidth}
		\includegraphics[width=\textwidth]{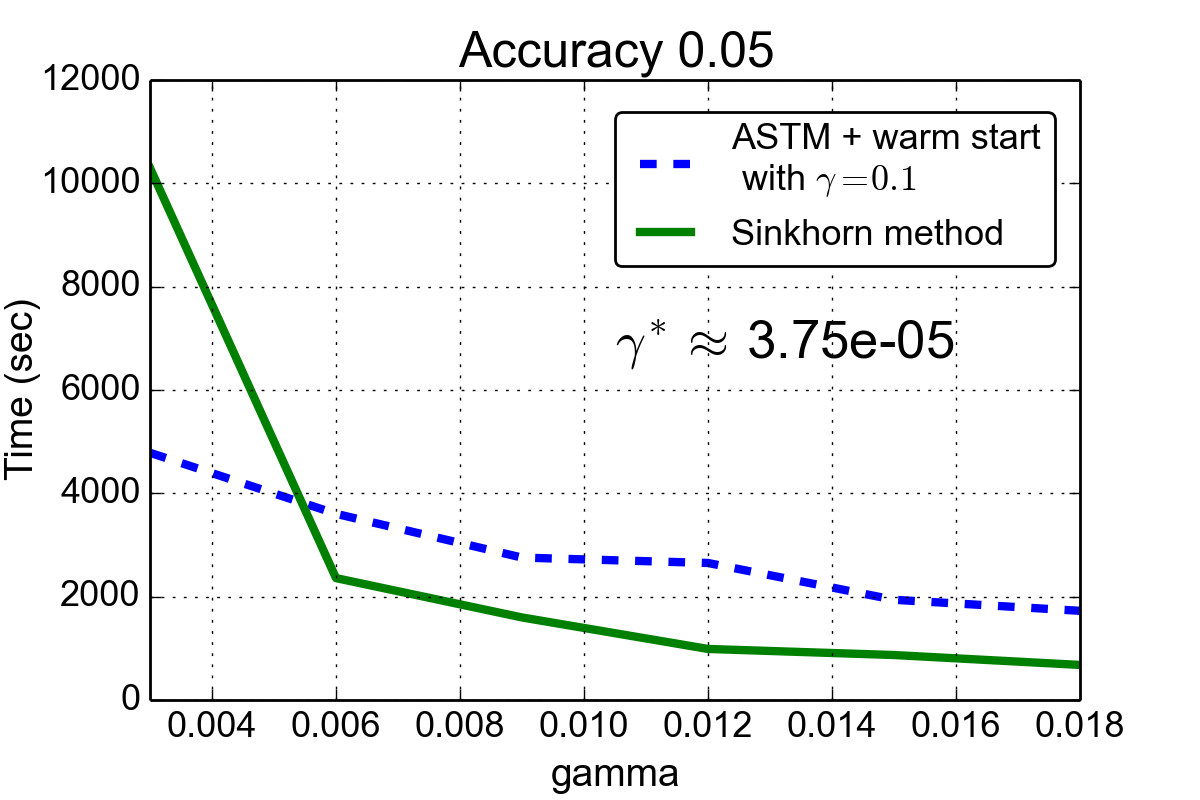}
	\end{subfigure}

	\centering
	\begin{subfigure}[b]{0.4\textwidth}
		\includegraphics[width=\textwidth]{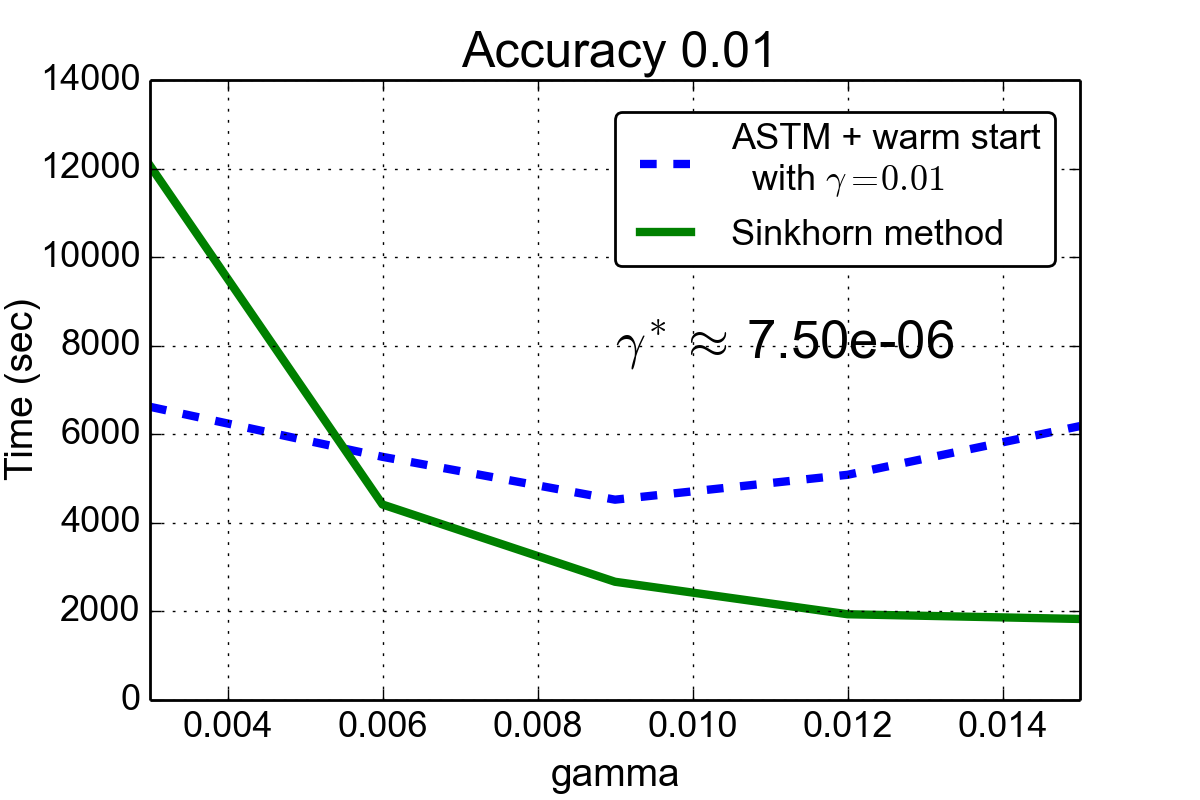}
	\end{subfigure}
	\caption{The perfomance of PDASTM with warm start vs Sinkhorn's method, Euclidean cost matrix $C$, MNIST dataset.}
	\label{fig:varwn_astm_MNIST}
\end{figure}

As we can see on all graphs, for small values of $\gamma$, namely, smaller than some threshold $\gamma_0$, our method outperforms the state-of-the-art Sinkhorn's method. 
Note that, from \cite{nesterov2005smooth}, it follows that, for very small values of $\gamma$, less than some threshold $\gamma*=\frac{\varepsilon}{4 \ln p}$, a good approximation of the solution to the problem \eqref{eq:ROT} can be obtained by solution of the linear programming problem corresponding to $\gamma=0$. We point these thresholds $\gamma*$ on the figures above.
It should be noted that the threshold $\gamma_0$ is larger than $\gamma*$. This means that it is better to use our method, but not some method for linear programing problems.

Finally, we investigate the dependence of running time of PDASTM with warm start on the problem dimension $p$. As we can see from the Figure \ref{fig:dim_dep}, the dependence is close to quadratic, which was expected from the theoretical bounds. Also this dependence is close to that of the Sinkhorn's method.

\begin{figure}[H]	
	\centering
	\includegraphics[width=0.4\textwidth]{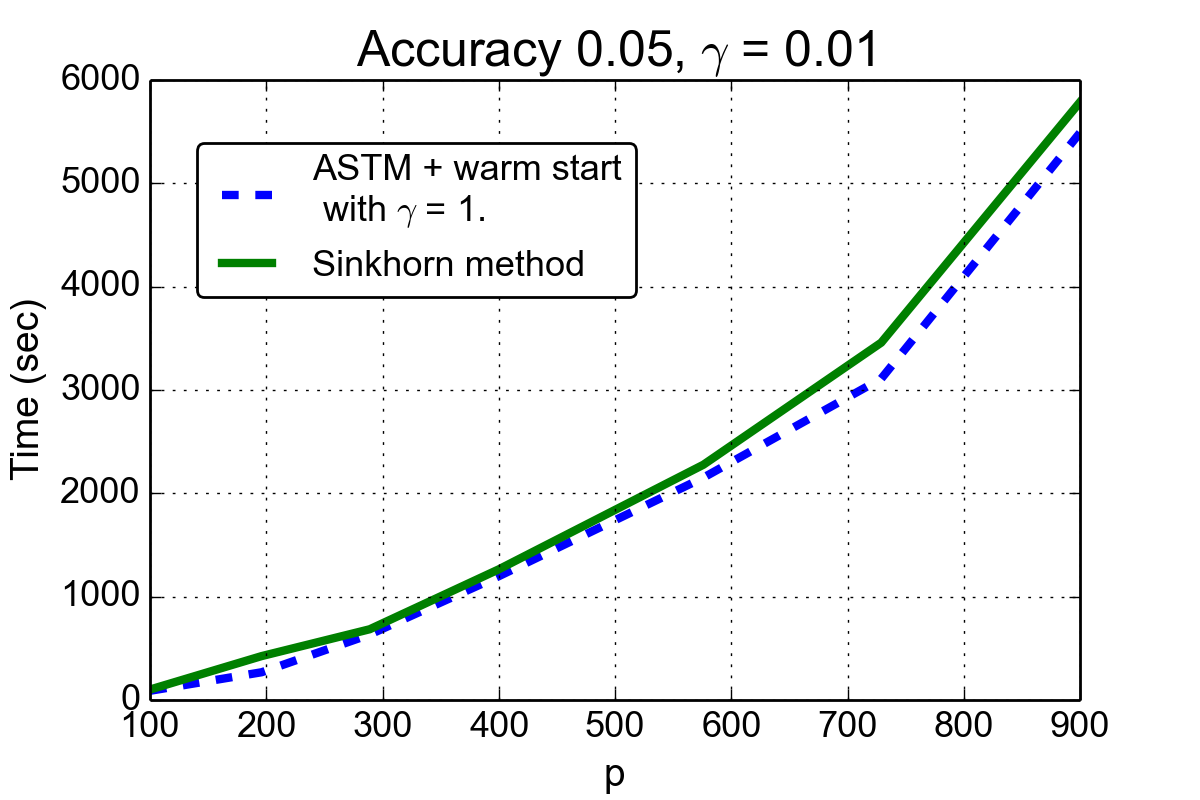}
	\caption{Dependence of running time from the problem dimension $p$.}
	\label{fig:dim_dep}
\end{figure}

\section*{Conclusion}
In this article, we propose a new adaptive accelerated gradient method for convex optimization problems and prove its convergence rate. We apply this method to a class of linearly constrained problems and show, how an approximate solution can be reconstructed. In the experiments, we consider two particular applied problems, namely, regularized optimal transport problem and traffic matrix estimation problem. The results of the experiments show that, in the regime of small regularization parameter, our algorithm outperforms the state-of-the-art Sinkhorn's-method-based approach.

\textbf{Acknowledgments.} The authors are very grateful to Yu. Nesterov for fruitful discussions. The work by A.Gasnikov was partially supported by RFBR grant 15-31-70001-mol\_a\_mos.

\bibliographystyle{plainnat}
\bibliography{references}


\end{document}